\documentclass[a4paper,12pt]{amsart}

\usepackage{amsmath}
\usepackage{amssymb}
\usepackage{amsthm}
\usepackage[a4paper]{geometry}
\usepackage{todonotes}
\usepackage{enumerate,mdwlist}   
\usepackage{cite}
\usepackage{mathrsfs}  

\usepackage{thmtools}
\usepackage{thm-restate}

\theoremstyle{plain}
\newtheorem{thm}{Theorem}[section]
\newtheorem*{thm*}{Theorem}

\newtheorem{lem}[thm]{Lemma}
\newtheorem{prop}[thm]{Proposition}

\theoremstyle{definition}

\newtheorem{exmp}{Example}[section]

\theoremstyle{remark}
\newtheorem{rem}{Remark}
\newcommand\inner[2]{\left\langle #1, #2 \right\rangle}

\newcommand{\tco}{\mathcal{S}^1}
\newcommand{\bo}{\mathcal{L}(L^2)}
\newcommand{\SC}{\mathcal{S}}

\newcommand{\compacts}{\mathcal{K}}

\newcommand{\C}{\mathbb{C}}
\newcommand{\R}{\mathbb{R}}
\newcommand{\Rd}{\mathbb{R}^d}
\newcommand{\Rdd}{\mathbb{R}^{2d}}

\newcommand{\N}{\mathbb{N}}

\newcommand{\Ld}{L^1(\Rd)}
\newcommand{\Ldd}{L^1(\Rdd)}

\newcommand{\HS}{\mathcal{S}^2}
\newcommand{\F}{\mathcal{F}}
\newcommand{\tr}{\mathrm{tr}}

\newcommand{\weyl}{a}
\author{Franz Luef and Eirik Skrettingland}
\title{A Wiener Tauberian theorem for operators and functions}
\address{Department of Mathematics\\ NTNU Norwegian University of Science and
Technology\\ NO–7491 Trondheim\\Norway}
\email{franz.luef@math.ntnu.no, eirik.skrettingland@ntnu.no}
\begin{document}

\begin{abstract}
    We prove variants of Wiener's Tauberian theorem in the framework of quantum harmonic analysis, i.e. for convolutions between an absolutely integrable function and a trace class operator, or of two trace class operators. Our results include Wiener's Tauberian theorem as a special case. Applications of our Tauberian theorems are related to localization operators, Toeplitz operators, isomorphism theorems between Bargmann-Fock spaces and quantization schemes with consequences for Shubin's pseudodifferential operator calculus and Born-Jordan quantization. Based on the links between localization operators and Tauberian theorems we note that the analogue of Pitt's Tauberian theorem in our setting implies compactness results for Toeplitz operators in terms of the Berezin transform. In addition, we extend the results on Toeplitz operators to other reproducing kernel Hilbert spaces induced by the short-time Fourier transform, known as Gabor spaces. Finally, we establish the equivalence of Wiener's Tauberian theorem and the condition in the characterization of compactness of localization operators due to Fern\'andez and Galbis.    
\end{abstract}
\maketitle
\section{Introduction}
In operator theory one views the space of trace class operators $\tco$ as the noncommutative analogue of the space of absolutely integrable functions $\Ld$ by viewing the trace of an operator as the substitute of the Lebesgue integral of a function. Over the years this point of view has led to a number of results in operator theory where one has extended concepts for functions to operators in an attempt to formulate operator-theoretic analogues of statements about functions. Guided by this meta-statement, Werner has proposed an operator-theoretic variant of harmonic analysis in \cite{Werner:1984}, which originated from his work in quantum physics and is thus referred to as ``quantum harmonic analysis".     

%Operator-theoretic analogs of Wiener's Tauberian theorem are the main results of this paper and these have various applications to theory of localization operators, Toeplitz operators and quantization schemes.

%Before we formulate our main results we review some basic facts from time-frequency analysis.
In this paper we establish a version of Wiener's Tauberian theorem in the setting of quantum harmonic analysis. Wiener's Tauberian theorem is a cornerstone of harmonic analysis. In short, it analyses the asymptotic properties of a bounded function by testing it with convolution kernels. %We denote by $\mathcal{W}$ the set of all $k\in\Ld$ With a non-vanishing Fourier transform $\widehat{k}$.
\begin{thm*}[Wiener's Tauberian Theorem]
Suppose $f\in L^\infty(\mathbb{R}^d)$ and $h\in\Ld$ with a non-vanishing Fourier transform $\widehat{h}$. Then the following implication holds for $A\in \mathbb{C}$: if 
 \begin{equation*}
    \lim_{x\to\infty} (h\ast f)(x)=A\int_{\mathbb{R}^d} h(y)\ dy,
 \end{equation*}
 then for any $g\in\Ld$ we have
 \begin{equation*}
    \lim_{x\to\infty} (g\ast f)(x)=A\int_{\mathbb{R}^d} g(y)\ dy.
 \end{equation*}
\end{thm*}
Moreover, Wiener noticed that the Tauberian condition holds \textit{only} for $h\in\Ld$  satisfying the condition $\widehat{h}(\omega)\ne 0$ for any $\omega\in\Rd$. The key step in the proof of this equivalence is based on the following approximation theorem. For $f\in\Ld$ we denote by $T_xf(t)=f(t-x)$ the translate of $f$ by $x\in\Rd$.
\begin{thm*}[Wiener's Approximation Theorem]
   For $f\in\Ld$ we have that $\overline{\text{span}}\{T_x f:\,x\in\Rd\}=\Ld$ if and only if $\widehat{f}(\omega)\ne 0$ for any $\omega\in\Rd$. 
\end{thm*}

In quantum harmonic analysis one complements the convolution $f\ast g(x)=\int_{\Rd} f(t)g(x-t)\ dt$ of $f,g\in\Ld$ with two new convolution operations: the convolution $f\star S$ of $f\in\Ld$ and a trace class operator $S$, and the convolution $S\star T$ of two trace class operators $S$ and $T$. This is achieved by replacing, for $z\in \mathbb{R}^{2d}$, the translation $T_z f$ of a function by the \textit{translation} $\alpha_z(R)$ of a bounded operator $R$ given by
\begin{equation*}
  \alpha_z(R)=\pi(z)R\pi(z)^*\quad \text{ for } z\in \Rdd,
\end{equation*}
where  $\left(\pi(z)\psi\right)(t)=e^{2\pi i  \omega\cdot t}\psi(x-t)$ denotes the time-frequency shift  of $\psi\in L^2(\Rd)$ by $z=(x,\omega)\in\Rdd$. 

For $f\in L^1(\Rdd)$ and $S\in \tco$, where $\tco$ denotes the trace class operators, the convolution $f\star S\in \tco$ is then defined by the Bochner integral
\begin{equation*} 
  f\star S := \int_{\Rdd} f(z) \alpha_z(S) \ dz,
\end{equation*}
which is another trace class operator. The convolution $S\star T$ of two operators $S,T\in \tco$ is the function
\begin{equation*} 
  S\star T(z) = \tr(S\alpha_z (\check{T}))\quad \text{ for } z\in \Rdd, 
\end{equation*}
where $\check{T}=PTP$, with $P$ the parity operator $P\psi(t)=\psi(-t).$ 

In summary, the convolutions $f\star S$ and $S\star T$ arise as extensions of the convolution of functions where one replaces either one or both $L^1$-functions with trace class operators. The seminal paper \cite{Werner:1984} contains a number of operator-theoretic versions of basic results from harmonic analysis, e.g. the Riemann-Lebesgue lemma, the Hausdorff-Young theorem and Wiener's approximation theorem. The
 variant of Wiener's approximation theorem in \cite{Werner:1984} concerns translates of a trace class operator being dense in the space of trace class operators, and is established by defining an operator-theoretic Fourier transform, the Fourier-Wigner transform $ \F_W(S)\in L^\infty(\Rdd)$ of a trace class operator $S$.

%Quantum harmonic analysis is based on the interplay of functions on phase space $\Rdd$ and trace class operators.

The appropriate Fourier transform for functions in $\Ldd$ is the symplectic Fourier transform $\F_\sigma$ and the following classes of functions and operators are going to be crucial in our Tauberian theorems for quantum harmonic analysis:
\begin{align*}
  W(\Rdd)&:=\{ f\in L^1(\Rdd):\F_\sigma(f)(z)\neq 0 \text{ for any } z\in \Rdd \}, \\
  \mathcal{W}&:= \{S\in \tco : \F_W(S)(z)\neq 0 \text{ for any }z\in \Rdd \}. 
\end{align*}
Our first main result is a generalization of Wiener's Tauberian Theorem for functions on $\Rdd$. Here $\compacts$ denotes the space of compact operators on $L^2(\Rd)$ and $I_{L^2}$ is the identity operator.
\begin{restatable*}[Tauberian theorem for bounded functions]{thm}{tauberianfunction}  \label{thm:tauberian1}
	Let $f\in L^\infty(\Rdd)$, and assume that one of the following equivalent statements holds for some $A\in \C$: 
		
	\begin{enumerate}[(i)]
		\item There is some $S\in \mathcal{W}$ such that 
	\begin{equation*}
  f\star S=A\cdot \tr(S)\cdot I_{L^2} +K 
\end{equation*}
for some compact operator $K\in \compacts$.
		\item There is some $a\in W(\Rdd)$ such that
	\begin{equation*}
  f\ast a=A\cdot \int_{\Rdd} a (z) \ dz+h 
\end{equation*}
for some $h\in C_0(\Rdd)$.
	\end{enumerate} 
	 Then both of the following statements hold:
	 \begin{enumerate}
	 	\item For any $T\in \tco$, $f\star T=A\cdot \tr(T)\cdot I_{L^2} +K_T$ for some compact operator $K_T\in \compacts$.
	 	\item For any $g\in L^1(\Rdd)$, $f\ast g=A\cdot\int_{\Rdd}g(z) \ dz +h_g$ for some $h_g\in C_0(\Rdd)$.
	 \end{enumerate} 
\end{restatable*}
We note that the equivalence $(ii)\iff (2)$ is Wiener's original Tauberian theorem. Similarly to Wiener's Tauberian theorem, this theorem concerns the asymptotic properties of the operator $R$ when we use the common intuition that asymptotic properties of an operator are properties that are invariant under compact perturbations, see \cite[Chap.3]{Arveson:2002}. There is also a version of the preceding theorem for bounded operators:
\begin{restatable*}[Tauberian theorem for bounded operators]{thm}{tauberianoperator}  \label{thm:tauberian2}
	Let $R\in \bo$, and assume that one of the following equivalent statements holds for some $A\in \C$: 
		
	\begin{enumerate}[(i)]
		\item There is some $S\in \mathcal{W}$ such that 
	\begin{equation*}
  R\star S=A\cdot \tr(S)+h 
\end{equation*}
for some $h\in C_0(\Rdd)$.
		\item There is some $a\in W(\Rdd)$ such that 
	\begin{equation*}
  R\star a=A\cdot \int_{\Rdd} a (z) \ dz\cdot I_{L^2}  +K 
\end{equation*}
for some compact operator $K\in \compacts$.
	\end{enumerate} 
	 Then both of the following statements hold:
	 \begin{enumerate}
	 	\item For any $T\in \tco$, $R\star T=A\cdot \tr(T)+h_T$ for some $h_T\in C_0(\Rdd)$.
	 	\item For any $g\in L^1(\Rdd)$, $R\star g=A\cdot\int_{\Rdd}g(z) \ dz \cdot I_{L^2}  +K_g$ for some compact operator $K_g\in \compacts$.
	 \end{enumerate} 
\end{restatable*}

These Tauberian theorems have numerous applications to localization operators, Toeplitz operators and quantization schemes. The link to localization operators allows us to add another equivalent assumption to Theorem \ref{thm:tauberian1}, formulated in terms of the short-time Fourier transform. Recall that the short-time Fourier transform $V_\phi\psi$ of $\psi$ for the window $\phi$ is given by $V_\phi\psi(z)=\langle \psi,\pi(z)\phi\rangle$. 
\begin{restatable*}{prop}{fernandezgalbis} \label{prop:fernandezgalbis}
Let $A\in \mathbb{C}$. Then
	$f\in L^\infty(\Rdd)$ satisfies the equivalent conditions $(i)$ and $(ii)$ in Theorem \ref{thm:tauberian1} if and only if 
	\begin{enumerate}[(iii)]
		\item There is some non-zero Schwartz function $\Phi$ on $\R^{2d}$ such that for every $R>0$ 
	 \begin{equation*} 
	\lim_{|x|\to \infty} \sup_{|\omega|\leq R} |V_\Phi (f-A)(x,\omega)| =0. 
\end{equation*}
	\end{enumerate}
\end{restatable*}
As condition $(ii )$ in Theorem \ref{thm:tauberian1} is the condition from Wiener's classical Tauberian theorem, condition $(iii)$ above, which first appeared in the context of localization operators in \cite{Fernandez:2006}, is a new characterization of the functions to which Wiener's classical Tauberian theorem applies. 

To be precise, the localization operator $\mathcal{A}^{\varphi_1,\varphi_2}_f$ with \textit{mask} $f\in L^\infty(\Rdd)$ and  \textit{windows} $\varphi_1,\varphi_2 \in L^2(\Rd)$, is defined by 
\begin{equation*}
  \mathcal{A}^{\varphi_1,\varphi_2}_f(\psi) = \int_{\Rdd} f(z) V_{\varphi_1}\psi(z)\pi(z) \varphi_2 \ dz.
\end{equation*}
The link from localization operators to Theorem \ref{thm:tauberian1} is then the simple relation $\mathcal{A}^{\varphi_1,\varphi_2}_f=f\star (\varphi_2\otimes \varphi_1)$, where $\varphi_2\otimes \varphi_1(\psi)=\inner{\psi}{\varphi_1}_{L^2}\varphi_2$. Localization operators are further linked to Toeplitz operators on \textit{Gabor spaces} $V_\varphi(L^2)$ -- which contain the Bargmann-Fock space as a special case -- this allows the study of Toeplitz operators using Theorem \ref{thm:tauberian1}.

% Let $\varphi \in L^2(\Rd)$ with $\|\varphi\|_{L^2}=1$. The short-time Fourier transform $$V_\varphi:L^2(\Rd) \to L^2(\Rdd)$$ is an isometry with adjoint operator $$V_\varphi^* F=\int_{\Rdd} F(z) \pi(z)\varphi \ dz \quad \text{ for } F\in L^2(\Rdd).$$  
 
 The \textit{Gabor space} associated with $\varphi$ with $\|\varphi\|_{L^2}=1$ is $V_\varphi(L^2):=V_\varphi(L^2(\Rd))\subset L^2(\Rdd)$. 
 %Moyal's identity yields  that $V_\varphi V_\varphi^*=\mathcal{P}_{V_\varphi (L^2)}$, where $\mathcal{P}_{V_\varphi (L^2)}$ denotes the orthogonal projection onto the subspace $V_\varphi (L^2)$ of $L^2(\Rdd)$. 
 The Gabor space $V_\varphi (L^2)$ is a reproducing kernel Hilbert space with reproducing kernel 
 \begin{equation*} 
  k^\varphi_z(z')=\inner{\pi(z)\varphi}{\pi(z')\varphi}_{L^2}=V_\varphi (\pi(z)\varphi)(z'),
\end{equation*}
for any $\psi \in L^2(\Rd)$. We will show that the intersection of different Gabor spaces is trivial whenever the windows are not scalar multiples of each other. Every $f\in L^\infty(\Rdd)$ then defines a \textit{Gabor Toeplitz} operator $T_f^\varphi:V_\varphi (L^2)\to V_\varphi (L^2)$ by 
\begin{equation*}
  T_f^\varphi(V_\varphi \psi) = \mathcal{P}_{V_\varphi (L^2)} (f\cdot V_\varphi \psi),
\end{equation*}
where $\mathcal{P}_{V_\varphi (L^2)}:L^2(\Rdd)\to V_\varphi(L^2)$ is the orthogonal projection.
%We will use the map $\Theta^\varphi:\mathcal{L}(V_\varphi(L^2))\to \bo$ defined by 
%\begin{equation} \label{eq:theta}
%  \Theta^\varphi(\tilde{T}):=V_{\varphi}^*\vert_{V_{\varphi}(L^2)}\tilde{T}V_\varphi \quad \text{ for } \tilde{T} \in %\mathcal{L}(V_\varphi(L^2)).
%\end{equation}
%to study Toeplitz operators.
%As $V_\varphi:L^2(\Rd)\to V_\varphi(L^2)$ is unitary, $\Theta^\varphi$ encodes a unitary equivalence, and is easily seen to be a linear, multiplicative and isometric isomorphism. We obtain the following well-known and easily verified result. 
It is well-known that $T_f^\varphi$ and $\mathcal{A}_f^{\varphi,\varphi}$ are unitarily equivalent. %$\mathcal{A}_f^{\varphi,\varphi}=\Theta^\varphi(T_f^\varphi)$.

If the window function $\varphi$ is the Gaussian $\varphi_0(x)=e^{-\pi x^2}$, then $V_{\varphi_0}(L^2)$ is, up to a simple unitary transformation, the space of entire functions on $\mathbb{C}^d$ known as the Bargmann-Fock space $\mathcal{F}^2(\C^d)$,
For every $F\in L^\infty(\mathbb{C}^d)$ one defines the \textit{Bargmann-Fock Toeplitz operator} $T^{\mathcal{F}^2}_F$ on $\mathcal{F}^2(\mathbb{C}^d)$ by
\begin{equation*}
  T^{\mathcal{F}^2}_F (H) = \mathcal{P}_{\mathcal{F}^2} (F\cdot H)
\end{equation*}
for any $H \in \mathcal{F}^2(\mathbb{C}^d)$. One has that if $f\in L^\infty(\Rdd)$ and $F\in L^\infty(\C^d)$ are related by $F(x+i\omega)=f(x,-\omega)$ the the following operators are unitarily equvialent:

	\begin{enumerate}
		\item The localization operator $\mathcal{A}^{\varphi_0,\varphi_0}_f:L^2(\Rd)\to L^2(\Rd)$.
		\item The Gabor Toeplitz operator $T_f^{\varphi_0}:V_{\varphi_0}(L^2) \to V_{\varphi_0}(L^2)$.
		\item The Bargmann-Fock Toeplitz operator $T_F^{\mathcal{F}^2}:\mathcal{F}^2(\mathbb{C}^d)\to \mathcal{F}^2(\mathbb{C}^d).$
	\end{enumerate}
Since $\mathcal{A}^{\varphi_0,\varphi_0}_f=f\star(\varphi_0\otimes \varphi_0)$, the equivalences above allow us to translate statements from convolutions of operators to Toeplitz operators. One of the results we translate to Toeplitz operators follows by noting that the Tauberian theorems concern compact perturbations of a scaling of the identity, i.e. operators $A\cdot I_{L^2}  +K$ for $0\neq A \in \mathbb{C}$ and $K\in \compacts$. Inspired by this --  without using the Tauberian theorem itself -- we apply Riesz' theory of such operators to obtain sufficient conditions for localization operators to be isomorphisms:
\begin{restatable*}{prop}{localizationisomorphism} \label{prop:isomorphisms}
Let $0\neq M\in \mathbb{R}$, $a\in L^\infty(\Rdd)$ and $\Delta \subset \Rdd$ a set of finite Lebesgue measure. Assume that the following assumptions hold:
\begin{enumerate}[(i)]
	\item $a(z)\geq -M$ for a.e. $z\in \Rdd$,
	\item $a(z) > -M$ for $z\notin \Delta$,
	\item $a$ satisfies assumption $(i)$ or $(ii)$ in Theorem \ref{thm:tauberian1} with $A=0.$
\end{enumerate}
Let $f=M+a$. Then $\mathcal{A}_{f}^{\varphi,\varphi}$ is an isomorphism on $L^2(\Rd)$ for any $0\neq \varphi\in L^2(\Rd).$
\end{restatable*}

We translate these results to the polyanalytic Bargmann-Fock space $\mathcal{F}^2_n(\C^d)$ for $n\in \N^d$ -- in particular $\mathcal{F}^2_0(\C^d)$ is the Bargmann-Fock space $\mathcal{F}^2(\C^d)$. 
\begin{restatable*}{prop}{BFisomorphism} \label{prop:BFisomorphism}
	\begin{enumerate}
		\item If $\Omega \subset \mathbb{C}^d$ satisfies that $\Omega^c$ has finite Lebesgue measure, then $T^{\mathcal{F}^2_n}_{\chi_\Omega}$ is an isomorphism on $\mathcal{F}^2_n(\C^d).$
		\item There is a real-valued, continuous $F\in L^\infty(\C^d)$ such that $\lim_{|z|\to \infty} |F(z)|$ does not exist, yet $T_{F}^{\mathcal{F}^2_n}$ is an isomorphism on $\mathcal{F}^2_n(\C^d).$
	\end{enumerate}
\end{restatable*}

Another class of our results concerns the \textit{Berezin transform}. For the Gabor space $V_\varphi (L^2)$ we can express the Berezin transform $\mathfrak{B}^\varphi:V_\varphi(L^2)\to L^\infty(\Rdd)$ as a convolution of operators. In particular, the Berezin transform of the Gabor Toeplitz operator $T_f^\varphi$ is simply a convolution of functions:
	\begin{equation*}
  \mathfrak{B}^\varphi T_f^\varphi(z)=\left( f\ast |V_\varphi \varphi|^2\right) (z).
\end{equation*}

Pitt's classical theorem gives a condition on $f\in L^\infty(\Rdd)$ that ensures that $f\ast g\in C_0(\Rdd)$ for $g\in W(\Rdd)$ implies $f\in C_0(\Rdd)$. In particular, this holds for uniformly continuous $f$. A natural analogue of uniformly continuous functions for operators is the set $$\mathcal{C}_1 :=\{R\in \bo: z\mapsto \alpha_z(R) \text{ is continuous from } \Rdd \text{ to } \bo\},$$ see \cite{Bekka:1990,Werner:1984}. Werner has obtained the following result in \cite{Werner:1984} which in light of our Tauberian theorem is an analogue of Pitt's theorem for operators.
\begin{restatable*}{thm}{wernerpitt} \label{thm:wernerpitt}
	Let $R\in \mathcal{C}_1$. The following are equivalent.
	\begin{itemize}
		\item $R\in \compacts$.
		\item $R\star S \in C_0(\Rdd)$ for some $S\in \mathcal{W}$.
		\item $R\star f\in \compacts$ for some $f\in W(\Rdd)$.
	\end{itemize}
\end{restatable*}
 Fulsche \cite{Fulsche:2019} has recently noted that the preceding theorem implies a result in \cite{Bauer:2012} for the Bargmann-Fock space. We show that the result holds for any Gabor space $V_\varphi(L^2)$ under certain conditions on $\varphi$. We would like to stress that it is a Pitt-type theorem for the Tauberian theorem for operators. 

\begin{restatable*}{thm}{axlerzheng} \label{thm:axlerzheng}
Let $\varphi\in L^2(\Rd)$ with $\|\varphi\|_{L^2}=1$ satisfy that $V_\varphi \varphi$ has no zeros, and let $\mathcal{T}^\varphi$ be the Banach algebra generated by Toeplitz operators $T^\varphi_f\subset \mathcal{L}(V_\varphi(L^2))$ for $f\in L^\infty(\Rdd)$. Then the following are equivalent for $\tilde{T}\in \mathcal{T}^\varphi$. 
\begin{itemize}
	\item $\tilde{T}$ is a compact operator on $V_\varphi(L^2)$.
	\item $\mathfrak{B}^\varphi \tilde{T}\in C_0(\Rdd).$
\end{itemize}
Furthermore, if $\tilde{T}=T_{f}^\varphi$ for some slowly oscillating $f\in L^\infty(\Rdd)$, then the conditions above are equivalent to $\lim_{|z|\to \infty}|f(z)|=0$.
\end{restatable*}

Examples of $\varphi$ satisfying that $V_\varphi \varphi$ has no zeros were recently investigated in \cite{Grochenig:2019}, for example the one-sided exponential. Hence these $\varphi$'s give different reproducing kernel Hilbert spaces $V_\varphi (L^2)$ such that Toeplitz operators are compact if and only if their Berezin transform vanishes at infinity.

The main result in \cite{Bauer:2012} follows in particular, as shown in \cite{Fulsche:2019}. We have added a statement on slowly oscillating functions that follows from the original version of Pitt's theorem.
\begin{restatable*}[Bauer, Isralowitz]{thm}{bauerisralowitz} 
Let $\mathcal{T}^{\mathcal{F}^2}$ be the Banach algebra generated by the Toeplitz operators $T^{\mathcal{F}^2}_F$ for $F\in L^\infty(\C^d)$. The following are equivalent for $\tilde{T}\in \mathcal{T}^{\mathcal{F}^2}$.

\begin{itemize}
	\item $\tilde{T}$ is a compact operator on $\mathcal{F}^2(\C^d).$
	\item $\mathfrak{B}^{\mathcal{F}^2}\tilde{T}\in C_0(\C^d).$
\end{itemize}	
If $\tilde{T}=T^{\mathcal{F}^2}_F$ for a slowly oscillating $F\in L^\infty(\mathbb{C}^d)$, then the conditions above are equivalent to $\lim_{|z|\to \infty}F(z)=0$.
\end{restatable*}
As a consequence we state a compactness result for Toeplitz operators.
\begin{restatable*}{cor}{toeplitzchar}
A Toeplitz operator $T^{\mathcal{F}^2}_F$ for $F\in L^\infty(\C^d)$ is a compact operator on $\mathcal{F}^2(\C^d)$ if and only if $$f\ast |V_{\varphi_0} \varphi_0|^2 \in C_0(\Rdd),$$ where $f(x,\omega)=F(x-i\omega)$ for $x,\omega \in \Rd$ and $|V_{\varphi_0} \varphi_0(z)|^2=e^{-\pi |z|^2}.$	
\end{restatable*}
Finally, Theorem \ref{thm:wernerpitt} gives a simple condition for compactness of localization operators in terms of the Gaussian $\varphi_0$.

\begin{restatable*}{prop}{localizationchar} \label{prop:charcompact}
	Let $f\in L^\infty(\Rdd)$ and $\psi_1,\psi_2\in L^2(\Rd)$. The localization operator $\mathcal{A}^{\psi_1,\psi_2}_f$ is compact if and only if 
	\begin{equation*}
  f\ast (V_{\varphi_0}\psi_2\overline{V_{\varphi_0}\psi_1}) \in C_0(\Rdd).
\end{equation*}
\end{restatable*}

Finally we recall from \cite{Luef:2018b} that any $R\in \bo$ defines a quantization scheme given by $f\mapsto f\star R$ for $f\in L^1(\Rdd)$ and a time-frequency distribution $Q_R$, given by sending $\psi \in L^2(\Rd)$ to $Q_R(\psi)(z) = (\psi \otimes \psi) \star \check{R}(z)$  for $z\in \Rdd$. The distribution $Q_R$ is of Cohen's class since we have 
  $Q_R(\psi)= \weyl_{\check{R}}\ast W(\psi,\psi)$,  where $\weyl_{\check{R}}$ is the Weyl symbol of $\check{R}$ and $W(\psi,\psi)$ the Wigner distribution of $\psi$.
  
  In the final section we deduce a statement relating compactness properties of the quantization scheme of $f\mapsto f\star R$ to properties of $Q_R(\psi)$.

\begin{restatable*}{prop}{quantization} \label{prop:quantization}
	Let $R\in \bo$. The following are equivalent.
	\begin{enumerate}[(i)]
		\item $Q_R(\varphi)\in C_0(\Rdd)$ for \textit{some} $\varphi \in L^2(\Rd)$ such that $V_\varphi \varphi$ has no zeros.
		\item $g\star R\in \compacts$ for \textit{some} $g\in W(\Rdd)$.
		\item $Q_R(\psi)\in C_0(\Rdd)$ for \textit{all} $\psi \in L^2(\Rd)$.
		\item $f\star R\in \compacts$ for \textit{all} $f\in L^1(\Rdd).$
	\end{enumerate} 
\end{restatable*}
Hence if one takes the Gaussian $\varphi_0$ for $(i)$, then checking if $Q_R(\varphi_0)\in C_0(\Rdd)$ provides a simple test for checking whether Conditions $(iii)$ and $(iv)$ hold. We apply this result to Shubin's $\tau$-quantization scheme and Born-Jordan quantization.

\subsection{Notations and conventions}
For topological vector spaces $X,Y$, we denote by $\mathcal{L}(X,Y)$ the set of continuous, linear operators from $X$ to $Y$. If $X=Y$ we write $\mathcal{L(X)}=\mathcal{L}(X,X).$ The space of compact operators on $L^2(\Rd)$ is denoted by $\compacts$. For $1\leq p < \infty$ we let $\SC^p$ denote the Schatten p-class of compact operators with singular values in $\ell^p$, and we use the convention that $\SC^\infty = \bo$. In particular, $\tco$ denotes the space of trace class operators on $L^2(\Rd)$, and the trace of a trace class operator $T\in \tco$ is denoted by $\tr(T)$. Also, $\SC^2$ is the space of Hilbert-Schmidt operators, which form a Hilbert space with respect to the inner product $\inner{S}{T}_{\HS}=\tr(ST^*).$

Given a topological vector space $X$ and its continuous dual $X'$, the action of $x^* \in X'$ on $y\in x$ is denoted by $\inner{x^*}{y}_{X',X}$. To agree with the Hilbert space inner product we use the convention that the duality bracket is linear in the first coordinate and antilinear in the second coordinate. The Schwartz functions on $\Rd$ are denoted by $\mathscr{S}(\Rd)$.

The Euclidean norm on $\mathbb{R}^d$ or $\mathbb{C}^d$ will be denoted by $|\cdot|$. For $\Omega \subset \Rd$, $\chi_\Omega$ denotes the characteristic function of $\Omega$. As usual, $C_0(\Rd)$ denotes the continuous functions on $\Rd$ vanishing at infinity, and we use $L^0(\Rd)$ to denote the space of measurable, bounded functions $f$ on  $\Rd$ such that $\lim_{|z|\to \infty} f(z)=0$,i.e. for every $\epsilon>0$ there is $R>0$ such that $|f(z)|<\epsilon$ for a.e. $|z|>R.$ We will refer to $L^p$-spaces on $\Rd,\Rdd$ and $\mathbb{C}^d$, and sometimes we will omit explicit reference to the underlying space when it is clear from the context, for instance by writing $\bo$ for $\mathcal{L}(L^2(\Rd))$. In all statements, measurability and "almost everywhere" properties will refer to Lebesgue measure.

\section{Preliminaries}
 
\subsection{Concepts from time-frequency analysis}

The mathematical theory of time-frequency analysis will provide the setup and many of the tools we use in this paper. We therefore introduce the \textit{time-frequency shifts} $\pi(z)\in \bo $ for $z=(x,\omega)\in \Rdd$, given by  $$\left(\pi(z)\psi\right)(t)=e^{2\pi i  \omega\cdot t}\psi(t-x) \quad \text{ for } \psi \in L^2(\Rd).$$ The time-frequency shift $\pi(z)$ is clearly given as a composition $\pi(z)=M_\omega T_x$ of a \textit{modulation operator} $M_\omega\psi(t)=e^{2\pi i \omega \cdot t}\psi(t)$ and a \textit{translation operator} $T_x\psi(t)=\psi(t-x).$ Given $\psi,\phi\in L^2(\Rd)$, the \textit{short-time Fourier transform} $V_\phi \psi$ of $\psi$ with window $\phi$ is the function on $\Rdd$ defined by $$V_{\phi}\psi(z)=\inner{\psi}{\pi(z) \phi}_{L^2}\quad \text{ for } z\in \Rdd.$$
The short-time Fourier transform satisfies the important orthogonality relation 
\begin{equation} \label{eq:moyal}
  \int_{\Rdd} V_{\phi_1}\psi_1(z)\overline{V_{\phi_2}\psi_2(z)} \ dz = \inner{\psi_1}{\psi_2}_{L^2}\inner{\phi_2}{\phi_1}_{L^2},
\end{equation}
 see \cite{Grochenig:2001,Folland:1989}, sometimes called Moyal's identity. Throughout this paper we will use $\varphi_0$ to denote the normalized Gaussian
\begin{equation*}
  \varphi_0 (t) = 2^{d/4}e^{-\pi t^2}\quad \text{ for } t\in \Rd,
\end{equation*}
and we will often refer to its short-time Fourier transform, which by \cite[Lem. 1.5.2]{Grochenig:2001} is given by
\begin{equation}\label{eq:stftgauss}
	V_{\varphi_0}\varphi_0(z)=e^{-\pi i x\cdot \omega}e^{-\pi |z|^2/2}\quad \text{ for } z=(x,\omega);
\end{equation}
the reader should note already at this point that $V_{\varphi_0}\varphi_0$ has no zeros.

\subsubsection{Wigner functions and the Weyl transform}

Given $\phi,\psi \in L^2(\Rd)$, a close relative of the short-time Fourier transform $V_{\phi}\psi$ is the \textit{cross-Wigner distribution} $W(\psi,\phi)$ defined by 
\begin{equation*}
  W(\psi,\phi)(x,\omega)=\int_{\Rd} \psi(x+t/2)\overline{\phi(x-t/2)}e^{-2\pi i \omega \cdot t} \ dt \quad \text{ for } (x,\omega) \in \Rdd.
\end{equation*}

The cross-Wigner distribution is the main tool needed to introduce the \textit{Weyl transform}, which associates to any $f\in \mathscr{S}'(\Rdd)$ an operator $L_f\in \mathcal{L}(\mathscr{S}(\Rd), \mathscr{S}'(\Rd))$ defined by requiring 
\begin{equation} \label{eq:weyldef}
  \inner{L_f(\psi)}{\phi}_{\mathscr{S}'(\Rd),\mathscr{S}(\Rd)}=\inner{f}{W(\phi,\psi)}_{\mathscr{S}'(\Rdd),\mathscr{S}(\Rdd)} \quad \text{ for all } \phi,\psi \in \mathscr{S}(\Rd).
\end{equation}
By the Schwartz kernel theorem \cite{Hormander:1983}, any $S\in \mathcal{L}(\mathscr{S}(\Rd), \mathscr{S}'(\Rd))$ is the Weyl transform $L_f$ for some unique $f\in \mathscr{S}'(\Rdd)$. We denote this $f$ by $\weyl_S$, and call it the \textit{Weyl symbol} of $S$. In other words, $S=L_{\weyl_S}$. Note that there is no relationship between boundedness of the function $f$ and boundedness of the operator $L_f$ on $L^2(\Rd)$: there is $f\in L^\infty(\Rdd)$ such that $L_f\notin \bo$, and there is $S\in \bo$ such that $\weyl_S\notin L^\infty(\Rdd).$ See Remark \ref{rem:weylboundedness} for examples.

\begin{exmp}[Rank-one operators] 
	Given $\psi,\phi\in L^2(\Rd)$, the rank-one operator $\psi\otimes \phi\in \bo$ is defined by $$(\psi\otimes \phi)(\xi)=\inner{\xi}{\phi}_{L^2}\psi \quad \text{ for } \xi\in L^2(\Rd).$$ It is well-known that the Weyl symbol of $\psi\otimes \phi$ is $W(\psi,\phi).$
\end{exmp}

\subsubsection{Localization operators}

For a \textit{mask} $f\in L^\infty(\Rdd)$ and a pair of \textit{windows} $\varphi_1,\varphi_2 \in L^2(\Rd)$, we define the \textit{localization operator} $\mathcal{A}^{\varphi_1,\varphi_2}_f(\psi)\in \bo $ by 
\begin{equation*}
  \mathcal{A}^{\varphi_1,\varphi_2}_f(\psi) = \int_{\Rdd} f(z) V_{\varphi_1}\psi(z)\pi(z) \varphi_2 \ dz,
\end{equation*}
where the integral is interpreted weakly in the sense that we require
\begin{equation} \label{eq:weaklocop}
  \inner{\mathcal{A}^{\varphi_1,\varphi_2}_f(\psi)}{\phi}_{L^2(\Rd)} = \inner{f}{V_{\varphi_2}\phi \overline{V_{\varphi_1}\psi}}_{L^2(\Rdd)} \quad \text{ for any } \psi,\phi \in L^2(\Rd).
\end{equation}
It is well-known that $\mathcal{A}^{\varphi_1,\varphi_2}_f$ is bounded on $L^2(\Rd)$ for $f\in L^\infty(\Rdd)$ and $\varphi_1,\varphi_2 \in L^2(\Rd)$ \cite{Cordero:2003}, but one may also define localization operators for other Banach function spaces of masks $f$ and windows $\varphi_1,\varphi_2$ by interpreting  the brackets in \eqref{eq:weaklocop} as duality brackets, see \cite{Cordero:2003}. We postpone this discussion until we have a more suitable framework, which we now introduce.

\subsection{Quantum harmonic analysis: convolutions of operators and functions} \label{sec:convolutions}
In this section we introduce the quantum harmonic analysis developed by Werner in \cite{Werner:1984}, the main concepts of which are convolutions of operators and functions and a Fourier transform of operators. For a more detailed introduction in our terminology we refer to \cite{Luef:2018c}. Given any $z\in \Rdd$ and an operator $R\in \bo$, we define the \textit{translation} $\alpha_z(R)$ of $R$ by $z$ to be the operator
\begin{equation*}
  \alpha_z(R)=\pi(z)R\pi(z)^*.
\end{equation*}
At the level of Weyl symbols, we have that $$\alpha_z(R)=L_{T_z(\weyl_R)},$$ hence $\alpha_z$ corresponds to a translation of the Weyl symbol.
For $f\in L^1(\Rdd)$ and $S\in \tco$ we then define the convolution $f\star S\in \tco$ by the Bochner integral
\begin{equation} \label{eq:defconvfunop}
  f\star S := S\star f:= \int_{\Rdd} f(z) \alpha_z(S) \ dz.
\end{equation}

Hence the convolution of a function with an operator is a new operator. The convolution $S\star T$ of two operators $S,T\in \tco$ is the \textit{function}
\begin{equation} \label{eq:defconvopop}
  S\star T(z) = \tr(S\alpha_z (\check{T}))\quad \text{ for } z\in \Rdd.
\end{equation}
Here $\check{T}=PTP$, with $P$ the parity operator $P\psi(t)=\psi(-t).$ Then $S\star T\in L^1(\Rdd)$ with $\int_{\Rdd} S\star T(z) \ dz = \tr(S)\tr(T)$ and $S\star T=T\star S$\cite{Werner:1984}. Taking convolutions with a fixed operator or function is easily seen to be a linear map. 

One of the most important properties of the convolutions \eqref{eq:defconvfunop} and \eqref{eq:defconvopop} is that they interact nicely with each other and with the usual convolution $f\ast g(x)=\int_{\Rd} f(t)g(x-t)\ dt$ of functions, as is most strikingly shown by their associativity\cite{Werner:1984,Luef:2018c}. 

\begin{prop}
	The convolutions \eqref{eq:defconvfunop} and \eqref{eq:defconvopop} are  associative. Written out in detail, this means that for $S,T,R\in \tco$ and $f,g\in L^1(\Rdd)$ we have
	\begin{align*}
		(R\star S)\star T &= R \star (S\star T)  \\
		f\ast (R\star S) &= (f\star R) \star T \\ 
		(f\ast g) \star R &= f\star (g\star R).
	\end{align*}
\end{prop}
\begin{rem}
Special cases of this associativity have appeared several times in the literature, typically with less transparent formulations and proofs than those allowed by the convolution formalism. See for instance \cite[Prop. 3.10]{Fernandez:2006}.	
\end{rem}
The convolutions also have an interesting interpretation in terms of the Weyl symbol, as we have that 
\begin{align}
	S\star T(z)&=\weyl_S \ast \weyl_T(z) \label{eq:opconvweyl} \\
	\weyl_{f\star S}(z)&=f\ast \weyl_S(z)\nonumber. 
\end{align}
As is shown in detail in \cite{Luef:2018c}, one can extend the domains of the convolutions by duality. For instance, the convolution $f\star S\in \bo$ of $S\in \tco$ and $f\in L^\infty(\Rdd)$ is defined by $$\inner{f\star S}{T}_{\bo,\tco}=\inner{f}{\check{S}^*\star T}_{L^\infty,L^1}.$$ Combining this with a complex interpolation argument gives a version of Young's inequality \cite{Werner:1984,Luef:2018c}. Recall our convention that $\SC^\infty=\bo.$

\begin{prop}[Young's inequality]\label{prop:young}
	Let $1\leq p,q,r \leq \infty$ be such that $\frac{1}{p}+\frac{1}{q}=1+\frac{1}{r}$. If $f\in L^p(\R^{2d}), S \in \SC^p$ and $T\in \SC^q$, then $f\star T\in \SC^r$ and $S\star T\in L^r(\Rdd)$ may be defined and satisfy the norm estimates
	\begin{align*}
  		\|f\star T\|_{\SC^r} &\leq \|f\|_{L^p} \|T\|_{\SC^q}, \\
  		\|S\star T\|_{L^r} &\leq \|S\|_{\SC^p} \|T\|_{\SC^q}.  		
	\end{align*}
\end{prop}
\begin{rem}
	It is worth noting that if $S\in \tco$ and $T\in \bo$, then $S\star T$ is still given by \eqref{eq:defconvopop}, which can be interpreted pointwise, so that $S\star T$ is a continuous, bounded function. 
\end{rem}

Young's inequality above shows that the convolutions interact in a predictable way with $L^p(\Rdd)$ and $\SC^q$. We now show that the same is true for functions vanishing at infinity and compact operators. Recall that $L^0(\Rdd)$ denotes the Banach subspace of $L^\infty(\Rdd)$ consisting of $f\in L^\infty(\Rdd)$ that vanish at infinity. The following result shows that convolutions with trace class operators interchange $L^0(\Rdd)$ and $\compacts,$ which is the basis for our main theorems. These results are known, in particular we mention that part $(ii)$ was proved for rank-one operators $S$ in \cite{Boggiatto:2004a} using essentially the same proof.
\begin{lem}\label{lem:compactL0}
Let $R\in \compacts$ and $f\in L^0(\Rdd).$ If $S\in \tco$, then
\begin{enumerate}[(i)]
	\item $R\star S \in C_0(\Rdd)$,
	\item $f\star S \in \compacts,$
\suspend{enumerate}
and if $a\in L^1(\Rdd)$ then 
\resume{enumerate}[{[(i)]}]
	\item $R\star a\in \compacts$,
	\item $f\ast a  \in C_0(\Rdd)$. 
\end{enumerate}
\end{lem}

\begin{proof}
	Part $(i)$ is \cite[Prop. 4.6]{Luef:2018c}. For $(ii)$ and $(iv)$, note that any $f\in L^0(\Rdd)$ is the limit in the norm topology of $L^\infty(\Rdd)$ of a sequence of compactly supported functions $f_n$ -- simply pick $f_n=f\cdot \chi_{B_n(0)},$ where $B_n(0)=\{z\in \Rdd:|z|<n\}$. Clearly $f_n\in L^1(\Rdd)$, hence $f_n\star S\in \tco \subset \compacts$. We therefore have by Young's inequality (recall that $\SC^\infty = \bo$): $$\|f\star S-f_n\star S\|_{\bo}=\|(f-f_n)\star S\|_{\bo}\leq \|f-f_n\|_{L^\infty} \|S\|_{\tco}\to 0 \text{ as } n\to \infty,$$ so $f\star S$ is the limit in the operator norm of compact operators, hence itself compact. Similarly, $f_n\ast a\in C_0(\Rdd)$ and  $f_n\ast a$ converges uniformly to $f\ast a$ by Young's inequality $\|(f-f_n)\ast a\|_{L^\infty}\leq \|f-f_n\|_{L^\infty} \|a\|_{L^1}$, so that $f\ast a  \in C_0(\Rdd).$ Finally, $(iii)$ follows by noting that any $R\in \compacts$ is the limit in the operator norm of a sequence $R_n \in \tco$ of finite-rank operators. Then $R_n \star a\in \tco$ is compact, so it follows by $\|(R-R_n)\star a\|_{\bo}\leq \|R-R_n\|_{\bo} \|a\|_{L^1}$ that $R\star a$ is the limit in the operator norm of a sequence of compact operators, hence itself compact. 
\end{proof} 
\begin{rem}\label{rem:sufficientcompact}
	In combination with Proposition \ref{prop:young} and the fact that $\SC^p\subset \compacts$ for $p<\infty$, we see that $L^p(\Rdd)\star \tco \subset \compacts$ for $p=0$ and $1\leq p <\infty$.
\end{rem}

Finally, the convolutions preserve identity elements\cite[Prop. 3.2 (3)]{Werner:1984}. Here $I_{L^2}\in \bo$ is the identity operator and $1\in L^\infty(\Rdd)$ is given by $1(z)=z$.
\begin{lem} \label{lem:convolutionswithidentity}
	Let $S\in \tco$ and $f\in L^1(\Rdd)$. Then 
	\begin{align*}
		S\star I_{L^2} &= \tr(S)\cdot 1, \\
		S\star 1 &= \tr(S)\cdot I_{L^2}, \\
		f\star I_{L^2} &= \int_{\Rdd} f(z) \ dz \cdot I_{L^2}, \\
		f\ast 1 &= \int_{\Rdd} f(z) \ dz \cdot 1.
	\end{align*}
\end{lem}

\subsubsection{Fourier transforms of functions and operators}

As our Fourier transform of functions on $\Rdd$ we will use the symplectic Fourier transform $\F_\sigma$, given, for $f\in L^1(\Rdd)$, by 
\begin{equation*}
\F_{\sigma} f(z)=\int_{\R^{2d}} f(z') e^{-2 \pi i \sigma(z,z')} \ dz' \quad \text{ for } z \in \Rdd,
\end{equation*} where $\sigma$ is the standard symplectic form $\sigma((x_1,\omega_1),(x_2, \omega_2))=\omega_1\cdot x_2-\omega_2 \cdot x_1$. Clearly  $\F_\sigma$ is related to the usual Fourier transform $\widehat{f}(z)=\int_{\Rdd} f(z') e^{-2\pi i z\cdot z'} \ dz'$ by  $$\F_\sigma(f)(x,\omega)=\widehat{f}(\omega,-x),$$ so $\F_\sigma$ shares most properties with $\widehat{f}$: it extends to a unitary operator on $L^2(\Rdd)$ and to a bijection on $\mathscr{S}'(\Rdd)$ -- see \cite{deGosson:2011}. In addition, $\F_\sigma$ is its own inverse: $\F_\sigma \circ \F_\sigma =I_{L^2} $.

We will also use a Fourier transform of operators, namely the Fourier-Wigner transform $\F_W$ introduced by Werner \cite{Werner:1984} (Werner calls it the Fourier-Weyl transform, our usage of Fourier-Wigner agrees with \cite{Folland:1989}). When $S\in \tco$, $\F_W(S)$ is the function 
\begin{equation} \label{eq:deffourierwigner}
	\F_W(S)(z)=e^{-\pi i x \cdot \omega} \tr(\pi(-z)S)\quad \text{ for } z=(x,\omega)\in \Rdd.
\end{equation}
 As is shown in \cite{Werner:1984,Luef:2018b}, $\F_W$ extends to a unitary mapping $\F_W:\HS \to L^2(\Rdd)$ and a bijection onto $\mathscr{S}'(\Rdd)$ from $\mathcal{L}(\mathscr{S}'(\Rd), \mathscr{S}(\Rd))$. 
 
 The Fourier transforms interact in the expected way with convolutions \cite{Werner:1984}: if $S,T\in \tco$ and $f\in L^1(\Rdd)$, then
 \begin{align}
  \F_\sigma(S\star T)&=\F_W(S)\cdot\F_W(T),\label{eq:fsigmaconv} \\
  \F_W(f\star S)&= \F_\sigma(f)\cdot \F_W(S)\label{eq:fwignerconv}.
\end{align}
 We may also connect $\F_W$ and $\F_\sigma$ by the Weyl transform. In fact, we have by \cite[Prop. 3.16]{Luef:2018b} that
\begin{equation} \label{eq:weyltransformfw}
  \F_W(L_f)=\F_\sigma(f) \text{ for } f\in \mathscr{S}'(\Rdd).
\end{equation}

A main concern for this paper will be functions and operators satisfying that the appropriate Fourier transform never vanishes. Following the notation of \cite{Korevaar:2004} for the function case, we introduce the following notation:
\begin{align*}
  W(\Rdd)&:=\{ f\in L^1(\Rdd):\F_\sigma(f)(z)\neq 0 \text{ for any } z\in \Rdd \}, \\
  \mathcal{W}&:= \{S\in \tco : \F_W(S)(z)\neq 0 \text{ for any }z\in \Rdd \}. 
\end{align*}

The key tool for proving the Tauberian theorem for operators is the following generalization of Wiener's approximation theorem, originally proved by Werner \cite{Werner:1984}. See also \cite{Kiukas:2012,Luef:2018c} for more general statements.

\begin{thm}[Werner]\label{thm:wernerapproximation}
	Let $S\in \tco$. The following are equivalent.
	
	\begin{enumerate}
		\item The  linear span of the translates $\{\alpha_z(S)\}_{z\in \Rdd}$ is dense in $\tco$.
		\item $S\in \mathcal{W}$.
		\item The set $L^1(\Rdd)\star S=\{f\star S:f\in L^1(\Rdd)\}$ is dense in $\tco$.
		\item The map $T\mapsto S\star T$ is injective from $\bo$ to $L^\infty(\Rdd)$. 
		\item The set $\tco \star S=\{T \star S:T\in \tco\}$ is dense in $L^1(\Rdd).$
		\item The map $f\mapsto f\star S$ is injective from $L^\infty(\Rdd)$ to $\bo.$
	\end{enumerate}
\end{thm}

\subsubsection{The special case of rank-one operators}
When $S\in \tco$ is a rank-one operator $\psi\otimes \phi$ for $\psi,\phi\in L^2(\Rd)$, then many of the concepts introduced above are familiar concepts from time-frequency analysis. First we note that by \cite[Thm. 5.1]{Luef:2018c}, localization operators $\mathcal{A}^{\varphi_1,\varphi_2}_f$ can be described as convolutions by
\begin{equation} \label{eq:locopsasconv}
  \mathcal{A}^{\varphi_1,\varphi_2}_f = f \star (\varphi_2 \otimes \varphi_1).
\end{equation}

 Other convolutions and Fourier-Wigner transforms of rank-one operators are summarized in the next lemma. See \cite[Thm. 5.1 and Lem. 6.1]{Luef:2018c} for proofs. Here $\check{\varphi}(t):=(P\varphi)(t)=\varphi(-t).$
 
\begin{lem} \label{lem:rankonecase}
	Let $\varphi_1,\varphi_2,\xi_1,\xi_2 \in L^2(\Rd)$ and $S\in \bo$. Then, for $(x,\omega)\in \Rdd$,
	\begin{enumerate}
		\item  $\F_W(\varphi_1\otimes \varphi_2)(x,\omega)=e^{i\pi x\cdot \omega} V_{\varphi_2}\varphi_1(x,\omega).$
		\item $S\star (\varphi_1\otimes \varphi_2)(z)=\inner{S\pi(z) \check{\varphi_1}}{\pi(z) \check{\varphi_2}}_{L^2}$.
		\item $(\xi_1 \otimes \xi_2)\star (\check{\varphi_1}\otimes \check{\varphi_2})(x,\omega)=V_{\varphi_2} \xi_1(x,\omega)\overline{V_{\varphi_1}\xi_2(x,\omega)}.$
	\end{enumerate}
	In particular, for $\xi,\varphi \in L^2(\Rd)$
	\begin{equation*}	
  (\xi \otimes \xi) \star (\check{\varphi}\otimes \check{\varphi})(z)=|V_\varphi \xi (z)|^2.
\end{equation*}
\end{lem}

\begin{exmp}[Standard Gaussian]  \label{exmp:gaussianfw}
	By \eqref{eq:stftgauss}, $\F_W(\varphi_0\otimes \varphi_0)(z)=e^{-\pi|z|^2/2}$. We point out this simple case as it shows that $\varphi_0\otimes \varphi_0 \in \mathcal{W}$. In particular, $\mathcal{W}$ is non-empty. 
\end{exmp}

\section{Toeplitz operators and Berezin transforms}

In this section we will introduce some families of reproducing kernel Hilbert spaces and the corresponding Toeplitz operators and Berezin transforms. We will relate these spaces and operators to the convolutions introduced in Section \ref{sec:convolutions}, which will later allow us to deduce results for reproducing kernel Hilbert spaces from the main results this paper. By far the most studied of the spaces we consider is the Bargmann-Fock space $\mathcal{F}^2(\C^d)$, and we will later investigate whether some well-known result for $\mathcal{F}^2(\C^d)$ can hold for other of the reproducing kernel Hilbert spaces we consider. 

\subsection{Gabor spaces $V_\varphi (L^2)$} 

 Let $\varphi \in L^2(\Rd)$ with $\|\varphi\|_{L^2}=1$. By \eqref{eq:moyal}, the short-time Fourier transform $$V_\varphi:L^2(\Rd) \to L^2(\Rdd)$$ is an isometry, and one easily confirms that its adjoint operator is $$V_\varphi^* F=\int_{\Rdd} F(z) \pi(z)\varphi \ dz \quad \text{ for } F\in L^2(\Rdd),$$ where the vector-valued integral is interpreted in a weak sense, see \cite[Sec. 3.2]{Grochenig:2001} for details. The \textit{Gabor space} associated with $\varphi$ is then the image $V_\varphi(L^2(\Rd))\subset L^2(\Rdd)$, which we denote by $V_\varphi(L^2)$ for brevity. One can show using \eqref{eq:moyal} that 
 \begin{align}
  V_\varphi^* V_\varphi&=I_{L^2(\Rd)} ,\nonumber \\ %\label{eq:reconstructionSTFT}
  V_\varphi V_\varphi^*&=\mathcal{P}_{V_\varphi (L^2)}, \label{eq:reproducing}
\end{align}

 where $\mathcal{P}_{V_\varphi (L^2)}$ denotes the orthogonal projection onto the subspace $V_\varphi (L^2)$ of $L^2(\Rdd)$. This means that $V_\varphi$ is a unitary operator from $L^2(\Rd)$ to $V_\varphi(L^2)$, with inverse $V_\varphi^*\vert_{V_{\varphi}(L^2)}$. By writing out the operators in \eqref{eq:reproducing} one deduces that $V_\varphi (L^2)$ is a reproducing kernel Hilbert space with reproducing kernel 
 \begin{equation} \label{eq:reproducingkernelGabor}
  k^\varphi_z(z')=\inner{\pi(z)\varphi}{\pi(z')\varphi}_{L^2}=V_\varphi (\pi(z)\varphi)(z'),
\end{equation}
 meaning that we have the reproducing formula
 \begin{equation*} %\label{eq:reproducingkernel}
  V_\varphi \psi(z) =\inner{V_\varphi \psi}{k^\varphi_z}_{L^2(\Rdd)} 
\end{equation*}
for any $\psi \in L^2(\Rd)$. Every $f\in L^\infty(\Rdd)$ then defines a \textit{Gabor Toeplitz} operator $T_f^\varphi:V_\varphi (L^2)\to V_\varphi (L^2)$ by 
\begin{equation*}
  T_f^\varphi(V_\varphi \psi) = \mathcal{P}_{V_\varphi (L^2)} (f\cdot V_\varphi \psi).
\end{equation*}
To study such Toeplitz operators in this paper, we will use the map 
\begin{align} \label{eq:theta}
\Theta^\varphi&:\mathcal{L}(V_\varphi(L^2))\to \bo \nonumber \\
  \Theta^\varphi(\tilde{T})&:=V_{\varphi}^*\vert_{V_{\varphi}(L^2)}\tilde{T}V_\varphi \quad \text{ for } \tilde{T} \in \mathcal{L}(V_\varphi(L^2)).
\end{align}
As $V_\varphi:L^2(\Rd)\to V_\varphi(L^2)$ is unitary, $\Theta^\varphi$ encodes a unitary equivalence, and is easily seen to be a linear, multiplicative and isometric isomorphism. We obtain the following well-known and easily verified result. 

\begin{prop} \label{prop:locoptoeplitz}
	Let $\varphi\in L^2(\Rd)$ with $\|\varphi\|_{L^2}=1$ and $f\in L^\infty(\Rdd)$. Then $$\mathcal{A}_f^{\varphi,\varphi}=\Theta^\varphi(T_f^\varphi).$$ In particular, $T_f^\varphi$ and $\mathcal{A}_f^{\varphi,\varphi}$ are unitarily equivalent.
\end{prop}

Now recall that in a reproducing kernel Hilbert space $\mathcal{H}$ consisting of functions on $\Rdd$ with normalized reproducing kernel $k_z$ for $z\in \Rdd$, the \textit{Berezin transform} $\mathfrak{B}\tilde{T}$ of a bounded operator $\tilde{T}\in \mathcal{L}(\mathcal{H})$ is the function $\Rdd \to \C$ defined by 
\begin{equation*}
	\mathfrak{B}\tilde{T}(z)=\inner{\tilde{T}k_z}{k_z}_{\mathcal{H}}.
\end{equation*}
For the Gabor space $V_\varphi (L^2)$ we can express the Berezin transform $\mathfrak{B}^\varphi:V_\varphi(L^2)\to L^\infty(\Rdd)$ as a convolution of operators.

\begin{lem} \label{lem:berezingabor}
	Let $\varphi \in L^2(\Rd)$ with $\|\varphi\|_{L^2}=1$, and let $\tilde{T}\in \mathcal{L}(V_\phi (L^2)).$  Then 
	\begin{equation*}
	\mathfrak{B}^\varphi \tilde{T}(z)=\Theta^\varphi(T)\star (\check{\varphi}\otimes \check{\varphi})(z).
	\end{equation*}
	In particular the Berezin transform of the Gabor Toeplitz operator $T_f^\varphi$ is
	\begin{equation*}
  \mathfrak{B}^\varphi T_f^\varphi(z)=\left( f\ast |V_\varphi \varphi|^2\right) (z).
\end{equation*}
\end{lem}
\begin{proof}
	Since $k_z^\varphi (z')=V_{\varphi}(\pi(z)\varphi)(z')$ by \eqref{eq:reproducingkernelGabor}, we have 
	\begin{align*}
		\Theta^\varphi(\tilde{T})\star (\check{\varphi}\otimes \check{\varphi})(z) &= \inner{\Theta^\varphi(\tilde{T})\pi(z)\varphi}{\pi(z)\varphi}_{L^2(\Rd)} \text{ by Lemma \ref{lem:rankonecase}} \\
		&= \inner{V_\varphi^*\tilde{T}V_\varphi(\pi(z)\varphi)}{\pi(z)\varphi}_{L^2(\Rd)} \quad \text{ by \eqref{eq:theta}} \\
		&= \inner{\tilde{T}V_\varphi(\pi(z)\varphi)}{V_\varphi(\pi(z)\varphi)}_{L^2(\Rdd)} \\
		&= \mathfrak{B}^\varphi \tilde{T}(z).
	\end{align*}
 Since Proposition \ref{prop:locoptoeplitz} and \eqref{eq:locopsasconv} give that 
 $$f\star (\varphi\otimes \varphi)=\mathcal{A}_f^{\varphi,\varphi}=\Theta^\varphi(T_f^\varphi ) ,$$
  we get from the first part that and associativity of convolutions that 
 \begin{align*}
  \mathfrak{B}^\varphi T_f^\varphi=\left[f\star (\varphi \otimes \varphi)\right]\star (\check{\varphi}\otimes \check{\varphi})=f\ast |V_\varphi \varphi|^2 \quad \text{ by Lemma \ref{lem:rankonecase}.}
\end{align*}
\end{proof}

\begin{rem}
	Gabor spaces and their relation to localization operators has been discussed in \cite{Hutnik:2010}, with emphasis on $f\in L^\infty(\Rdd)$ depending only on $x.$ The reproducing kernel $k^\varphi_z$ has also been studied as the kernel of determinantal point processes called \textit{Weyl-Heisenberg ensembles} \cite{Abreu:2017whe,Abreu:2019}.
\end{rem}

\subsubsection{Gabor spaces with different windows}

Having introduced the Gabor spaces $V_\varphi(L^2)$, we naturally ask whether the properties of $V_\varphi (L^2)$ as a reproducing kernel Hilbert space depend on the window $\varphi$ in an essential way. As a first result in this direction, we note that the intersection of different Gabor spaces is trivial whenever the windows are not scalar multiples of each other, first proved with different methods in \cite[Thm. 4.2]{Ghandehari:2014}. 
\begin{lem}
	Let $\varphi_1,\varphi_2\in L^2(\Rd)$ with $\|\varphi_1\|_{L^2}=\|\varphi_2\|_{L^2}=1$. If there exists $c\in \mathbb{C}$ such that $\varphi_1=c\varphi_2$, then $V_{\varphi_1}(L^2)=V_{\varphi_2}(L^2)$. Otherwise $V_{\varphi_1}(L^2)\cap V_{\varphi_2}(L^2)=\{0\}$.
\end{lem}  

\begin{proof}
	If $\varphi_1=c\varphi_2$, then $V_{\varphi_1}\xi=V_{\varphi_2}(\overline{c}\xi)$, which implies the first part. Then assume that $0\neq V_{\varphi_1}\xi=V_{\varphi_2}\psi$ for $\xi,\psi \in L^2(\Rd)$. It follows by Lemma \ref{lem:rankonecase} that \begin{equation*} 
  \xi \otimes \varphi_1=\psi \otimes \varphi_2,
\end{equation*}
  as $\F_W$ is a bijection from $\HS$ to $L^2(\Rdd)$. Taking adjoints, we get 
 \begin{equation} \label{eq:proofgaborspaces}
  \varphi_1 \otimes \xi=\varphi_2 \otimes \psi.
\end{equation}
 
 If we apply \eqref{eq:proofgaborspaces} to $\xi$, we obtain $$\varphi_1=\frac{\inner{\xi}{\psi}_{L^2}}{\|\xi\|_{L^2}^2}\varphi_2.$$ Note that dividing by $\|\xi\|_{L^2}^2$ is allowed, as we assumed $V_{\varphi_1}\xi \neq 0$ which by \eqref{eq:moyal} implies $\xi\neq 0.$
\end{proof}

Even though the result above shows that Gabor spaces with different windows $\varphi_1$ and $\varphi_2$ usually have trivial intersection, there is always an obvious Hilbert space isomorphism $\Psi:V_{\varphi_1}(L^2)\to V_{\varphi_2}(L^2)$ given by $\Psi=V_{\varphi_2}V_{\varphi_1}^*\vert_{V_{\varphi_1} (L^2)}.$ However, this does not preserve the reproducing kernels:  $k^{\varphi_1}_z=V_{\varphi_1}(\pi(z)\varphi_1)$ by \eqref{eq:reproducingkernelGabor}, so clearly $\Psi(k^{\varphi_1}_z)=V_{\varphi_2}(\pi(z)\varphi_1)$. By the injectivity of $V_{\varphi_2}$, the only way $\Psi(k^{\varphi_1}_z)=V_{\varphi_2}(\pi(z)\varphi_1)$ can equal $k^{\varphi_2}_z=V_{\varphi_2}(\pi(z)\varphi_2)$ is if $\varphi_1=\varphi_2$.

If we use Proposition \ref{prop:locoptoeplitz} and Lemma \ref{lem:berezingabor} to translate parts of Theorem \ref{thm:wernerapproximation} into a result on Toeplitz operators, we clearly see that the properties of the window $\varphi$ must be taken into account when studying Toeplitz operators on $V_\varphi(L^2).$

\begin{prop} \label{prop:wernerapproxGabor}
	Let $\varphi\in L^2(\Rd)$ with $\|\varphi\|_{L^2}=1$. The following are equivalent.
	\begin{enumerate}
		\item $V_{\varphi}\varphi$ has no zeros. 
		\item The Berezin transform $\mathfrak{B}^\varphi$ is injective on $\mathcal{L}(V_{\varphi}(L^2)).$
		\item The map $f\mapsto T^\varphi_f$ is injective from $L^\infty(\Rdd)$ to $\mathcal{L}(V_{\varphi}(L^2)) .$
	\end{enumerate}
\end{prop}
\begin{proof}
The result will follow from Theorem \ref{thm:wernerapproximation} once we have shown that each statement is equivalent to a statement in that theorem with $S=\varphi\otimes \varphi$. As $\F_W(S)(x,\omega)=e^{i\pi x\cdot \omega}V_{\varphi}\varphi(x,\omega)$ by Lemma \ref{lem:rankonecase}, $(1)$ states that $S\in \mathcal{W}$. Since Proposition \ref{prop:locoptoeplitz} gives that $T_f^\varphi$ is unitarily equivalent with $\mathcal{A}_f^{\varphi,\varphi}=f\star S$, the map $f\mapsto T^\varphi_f$ is injective if and only if the map $f\mapsto f\star S$ is injective. Similarly, since 	Lemma \ref{lem:berezingabor} gives that
	 \begin{equation*}
	\mathfrak{B}^\varphi \tilde{T}(z)=\Theta^\varphi(\tilde{T})\star \check{S}
	\end{equation*}
	and $\Theta^\varphi:\mathcal{L}(V_\varphi(L^2))\to \bo$ is a bijection, we get that $\mathfrak{B}^\varphi$ is injective if and only if $T\mapsto T\star \check{S}$ is injective. It is simple to check that the last condition is equivalent to $T\mapsto T\star S$  being injective, as a calculation shows that $T\star \check{S}(z)=\check{T}\star S(-z).$  
	\end{proof}
	\begin{rem}
		The other parts of Theorem \ref{thm:wernerapproximation} could also be translated into statements on $V_{\varphi}(L^2)$, and one could obtain other equivalences by imposing weaker requirements on the set of zeros of $V_{\varphi}\varphi$, see \cite{Kiukas:2012,Luef:2018c}.
	\end{rem}
\subsection{Toeplitz operators on Bargmann-Fock space}

For the Gaussian $\varphi_0$, the Gabor space $V_{\varphi_0}(L^2)$ is closely related to another much-studied reproducing kernel Hilbert space: the Bargmann-Fock space $\mathcal{F}^2(\C^d)$, consisting of all analytic functions $F$ on $\mathbb{C}^d$ such that $ \|F\|_{\mathcal{F}^2}<\infty$, where $ \|F\|_{\mathcal{F}^2}$ is the norm induced by the inner product
\begin{equation*}
  \inner{F}{G}_{\mathcal{F}^2}=\int_{\C^d} F(z) \overline{G(z)} e^{-\pi |z|^2} \ dz .
\end{equation*}
  
 An important tool in the study of $\mathcal{F}^2(\C^d)$ is the \textit{Bargmann transform}, which is the unitary mapping $\mathcal{B}:L^2(\Rd)\to \mathcal{F}^2(\mathbb{C}^d)$ defined by
 \begin{equation} \label{eq:bargmannstft}
  \mathcal{B}=\mathcal{A}\circ V_{\varphi_0},
\end{equation}
where $\mathcal{A}:L^2(\Rdd)\to L^2(\mathbb{C}^d,e^{-\pi |z|^2}dz)$ is a unitary operator given by 
\begin{equation*}
  \mathcal{A}(f)(x+i\omega)=e^{-\pi i x \cdot \omega} e^{\frac{\pi}{2} |z|^2} f(x,-\omega)\quad \text{ for } z=(x,\omega) \in \Rdd.
\end{equation*}
The restriction $\mathcal{A}\vert_{V_{\varphi_0}(L^2)}$ is unitary from $V_{\varphi_0}(L^2)$ to $\mathcal{F}^2(\C^d)$, as it may be written as the composition $\mathcal{B}\circ V_{\varphi_0}^*\vert_{V_{\varphi_0}(L^2)}$ of unitary operators. Hence $\mathcal{A}$ allows us to relate the spaces $V_{\varphi_0} (L^2)$ and $\mathcal{F}^2(\C^d).$ 

The orthogonal projection from $L^2(\mathbb{C}^d,e^{-\pi |z|^2}dz)$ to $\mathcal{F}^2(\mathbb{C}^d)$ is given by 
\begin{equation} \label{eq:projectionBF}
  \mathcal{P}_{\mathcal{F}^2} = \mathcal{B}\mathcal{B}^* = \mathcal{A}V_{\varphi_0} V_{\varphi_0}^* \mathcal{A}^*=\mathcal{A}\mathcal{P}_{V_{\varphi_0}(L^2)}\mathcal{A}^*,
\end{equation}
and the non-normalized reproducing kernel of $\mathcal{F}^2(\mathbb{C}^d)$ is $$K_{z}(z')=e^{\pi \overline{z}\cdot z'}\quad \text{ for } z,z'\in \mathbb{C}^d.$$ For our purposes it is convenient to note that we can use the reproducing kernel $k^{\varphi_0}_{(x,\omega)}$ for $V_{\varphi_0}(L^2)$ to express $K_{z}$ for $z=x+i\omega$ by
\begin{equation} \label{eq:reproducingkernelfock}
  K_{z}(x'+i\omega')=e^{i\pi x\cdot \omega}e^{\pi |z|^2/2} \left[\mathcal{A}k^{\varphi_0}_{(x,-\omega)}\right](x'+i\omega'),
\end{equation}
as follows from the calculation 
\begin{align*}
  \inner{\mathcal{B}(\psi)}{e^{i\pi x\cdot \omega}e^{\pi |z|^2/2} \mathcal{A}k^{\varphi_0}_{(x,-\omega)}}_{\mathcal{F}^2}&=e^{-\pi i x\cdot \omega}e^{\pi |z|^2/2}\inner{\mathcal{A}V_{\varphi_0}\psi}{\mathcal{A}k^{\varphi_0}_{(x,-\omega)}}_{\mathcal{F}^2} \\
  &=e^{-\pi i x\cdot \omega}e^{\pi |z|^2/2}\inner{V_{\varphi_0}\psi}{k^{\varphi_0}_{(x,-\omega)}}_{L^2(\Rdd)} \\
  &= e^{-\pi i x\cdot \omega}e^{\pi |z|^2/2} V_{\varphi_0}\psi(x,-\omega)\\
  &=\mathcal{B}(\psi)(x+i\omega).
\end{align*}

For every $F\in L^\infty(\mathbb{C}^d)$ one defines the \textit{Bargmann-Fock Toeplitz operator} $T^{\mathcal{F}^2}_F$ on $\mathcal{F}^2(\mathbb{C}^d)$ by
\begin{equation*}
  T^{\mathcal{F}^2}_F (H) = \mathcal{P}_{\mathcal{F}^2} (F\cdot H)
\end{equation*}
for any $H \in \mathcal{F}^2(\mathbb{C}^d)$. Using \eqref{eq:projectionBF} and the unitarity of $\mathcal{A}$, one can calculate that if $f\in L^\infty(\Rdd)$ and $F\in L^\infty(\C^d)$ are related by 
\begin{equation} \label{eq:complexreal}
  F(x+i\omega)=f(x,-\omega) \quad \text{ for } x,\omega \in \Rdd,
\end{equation}
then 
\begin{equation} \label{eq:toeplitzgaborfock}
  T_f^{\varphi_0}=\mathcal{A}^* T_F^{\mathcal{F}^2} \mathcal{A}.
\end{equation}

 In combination with Proposition \ref{prop:locoptoeplitz} this gives the following result.

\begin{prop}\label{prop:unitarilyequivalent}
	Let $f\in L^\infty(\Rdd)$ and $F\in L^\infty(\mathbb{C}^d)$ be related by \eqref{eq:complexreal}. Then the following operators are unitarily equivalent.
	
	\begin{enumerate}
		\item The localization operator $\mathcal{A}^{\varphi_0,\varphi_0}_f:L^2(\Rd)\to L^2(\Rd)$.
		\item The Gabor Toeplitz operator $T_f^{\varphi_0}:V_{\varphi_0}(L^2) \to V_{\varphi_0}(L^2)$.
		\item The Bargmann-Fock Toeplitz operator $T_F^{\mathcal{F}^2}:\mathcal{F}^2(\mathbb{C}^d)\to \mathcal{F}^2(\mathbb{C}^d).$
	\end{enumerate}
\end{prop}

\begin{rem}
	The simple result above is far from new, going back to at least \cite{Coburn:2000}. A related and more complicated question that appears in the literature is to relate $\mathcal{A}^{\varphi,\varphi}_f$, where $\varphi$ needs no longer be Gaussian, to a Bargmann-Fock Toeplitz operator $T_{(I+D)F}^{\mathcal{F}^2}$,  where $D$ is some differential operator \cite{Abreu:2015,Coburn:2000,Englis:2009}. 
\end{rem}

The Berezin transform can also be defined on $\mathcal{F}^2(\mathbb{C}^d).$ Since $\mathcal{A}:V_\varphi (L^2)\to \mathcal{F}^2(\C^d)$ is unitary, one easily checks using \eqref{eq:reproducingkernelfock} that the \textit{normalized} reproducing kernel $\tilde{k}_{z}$ on $\mathcal{F}^2(\mathbb{C}^d)$ is $$\tilde{k}_{z}(z')=e^{i\pi x\cdot \omega} \left[\mathcal{A}k_{(x,-\omega)}^{\varphi_0}\right](x'+i\omega')\quad \text{ for } z=x+i\omega, z'=x'+i\omega'.$$ 

This implies the following result on the Berezin transform $\mathfrak{B}^{\mathcal{F}^2}$ on $\mathcal{F}^2(\mathbb{C}^d)$. 
\begin{lem}\label{lem:berezinfock} 
Let $\tilde{T}\in \mathcal{L}(\mathcal{F}^2(\mathbb{C}^d))$. Then 
\begin{align*}
\mathfrak{B}^{\mathcal{F}^2}\tilde{T}(x+i\omega)&=\mathfrak{B}^{\varphi_0}[\mathcal{A}^*\tilde{T}\mathcal{A}](x,-\omega)\\
&=(\mathcal{B}^*\tilde{T}\mathcal{B})\star (\varphi_0\otimes \varphi_0)(x,-\omega).  
\end{align*}
In particular, if $F\in L^\infty(\mathbb{C}^d)$, then
	\begin{equation*}
  \mathfrak{B}^{\mathcal{F}^2}T^{\mathcal{F}^2}_F(x+i\omega)=\left(f\ast |V_{\varphi_0} \varphi_0|^2 \right)(x,-\omega),
\end{equation*}
 where $f\in L^\infty(\Rdd)$ is given by $f(x,\omega)=F(x-i\omega)$ and $|V_{\varphi_0} \varphi_0(z)|^2=e^{-\pi |z|^2}.$
\end{lem}

\begin{proof}
	By definition, 
	\begin{align*}
  \mathfrak{B}^{\mathcal{F}^2}\tilde{T}(x+i\omega)&=\inner{\tilde{T}\tilde{k}_{x+i\omega}}{\tilde{k}_{x+i\omega}}_{\mathcal{F}^2} \\
  &= \inner{\tilde{T}\mathcal{A}k_{(x,-\omega)}^{\varphi_0}}{\mathcal{A}k_{(x,-\omega)}^{\varphi_0}}_{\mathcal{F}^2} \\
  &= \inner{\mathcal{A}^*\tilde{T}\mathcal{A}k_{(x,-\omega)}^{\varphi_0}}{k_{(x,-\omega)}^{\varphi_0}}_{L^2(\Rdd)}\\
  &=\mathfrak{B}^{\varphi_0}[\mathcal{A}^*\tilde{T}\mathcal{A}](x,-\omega).
\end{align*}
That this last expression equals $(\mathcal{B}^*\tilde{T}\mathcal{B})\star (\varphi_0\otimes \varphi_0)(x,-\omega)$ follows from Lemma \ref{lem:berezingabor}, since $\mathcal{B}^*\tilde{T}\mathcal{B}=V_{\varphi_0}^* [\mathcal{A}^* \tilde{T}\mathcal{A}]V_{\varphi_0}$. For the formula for Toeplitz operators, combine the first part with \eqref{eq:toeplitzgaborfock} and the final part of Lemma \ref{lem:berezingabor}.
\end{proof}

The results above show the intimate connection between $\mathcal{F}^2(\C^d)$ and the Gabor space $V_{\varphi_0}(L^2).$ Many of the results known for $\mathcal{F}^2(\C^d)$ can easily be translated into results for $V_{\varphi_0}(L^2)$, and we will later investigate certain conditions on $\varphi$ that allow us to generalize these results to other Gabor spaces $V_\varphi(L^2)$.

\subsection{Polyanalytic Bargmann-Fock spaces}
By \eqref{eq:bargmannstft}, we may identify $V_{\varphi_0}(L^2)$ and the Bargmann-Fock space by the operator $\mathcal{A}:L^2(\Rdd) \to L^2(\C^d,e^{-\pi |z|^2}dz)$. If the Gaussian $\varphi_0$ is replaced by another Hermite function $\varphi_n$ for $n\in \N^d$, and we define the \textit{polyanalytic Bargmann transform} $\mathcal{B}_n:L^2(\Rd) \to L^2(\mathbb{C}^d,e^{-\pi |z|^2}dz)$ by $$\mathcal{B}_n = \mathcal{A}\circ  V_{\varphi_n},$$ then the image of $\mathcal{B}_n$, which we denote by $\mathcal{F}^2_n$, is again a reproducing kernel Hilbert space with reproducing kernel $K^{\varphi_n}_{z}$ for $z=x+i\omega$ given by
\begin{equation*}% \label{eq:reproducingkernelpoly}
  K^{\varphi_n}_{z}(x'+i\omega')=e^{i\pi x\cdot \omega}e^{\pi |z|^2/2} \left[\mathcal{A}k^{\varphi_n}_{(x,-\omega)}\right](x'+i\omega').
\end{equation*}

Unlike the Bargmann-Fock space $\mathcal{F}^2=\mathcal{F}^2_0$, $\mathcal{F}^2_n$ does not in general consist of analytic functions, but rather of so-called polyanalytic functions. For this reason $\mathcal{F}^2_n$ is sometimes called the \textit{true polyanalytic Fock space of degree $n$} \cite{Abreu:2010,Abreu:2014,Balk:1991}. Following \cite{Keller:2019,Rozenblum:2019} we define, given $F\in L^\infty(\C^d)$, the \textit{polyanalytic Toeplitz operator} $T_F^{\mathcal{F}^2_n}:\mathcal{F}^2_n \to \mathcal{F}^2_n$ by $$T_F^{\mathcal{F}^2_n}(H) = \mathcal{P}_{\mathcal{F}^2_n} (F\cdot H)$$ for $H \in \mathcal{F}^2_n$. Similarly to Bargmann-Fock space the orthogonal projection $\mathcal{P}_{\mathcal{F}^2_n}$ from $L^2(\mathbb{C}^d,e^{-\pi |z|^2}dz)$ to $\mathcal{F}^2_n$ is given by $$\mathcal{P}_{\mathcal{F}^2_n}=\mathcal{B}\mathcal{B}^* = \mathcal{A}V_{\varphi_n} V_{\varphi_n}^* \mathcal{A}^*.$$ If $f\in L^\infty(\Rdd)$ and $F\in L^\infty(\C^d)$ are related as in \eqref{eq:complexreal}, one can show that $T_f^{\varphi_n}=\mathcal{A}^*T_F^{\mathcal{F}^2_n}\mathcal{A}$. Hence we obtain the following result. 

\begin{prop} \label{prop:uneqpolyana}
	Let $f\in L^\infty(\Rdd)$ and $F\in \C^d$ be related as in \eqref{eq:complexreal}. For $n\in \N^d$, the following operators are unitarily equivalent.
	
	\begin{enumerate}
		\item The localization operator $\mathcal{A}^{\varphi_n,\varphi_n}_f:L^2(\Rd)\to L^2(\Rd)$.
		\item The Gabor Toeplitz operator $T_f^{\varphi_n}:V_{\varphi_n}L^2 \to V_{\varphi_n}L^2$.
		\item The polyanalytic Toeplitz operator $T_F^{\mathcal{F}^2_n}:\mathcal{F}^2_n(\mathbb{C}^d)\to \mathcal{F}^2_n(\mathbb{C}^d).$
	\end{enumerate}
\end{prop}
We have related polyanalytic Toeplitz operators to Gabor Toeplitz operators on $V_{\varphi_n}(L^2).$ By \cite[(4.16)]{Janssen:1997}, $V_{\varphi_n}\varphi_n$ has zeros if and only if $n\neq 0$. An easy argument using the previous proposition then translates Proposition \ref{prop:wernerapproxGabor} into the following statement. A version of this is also discussed with different tools in \cite[Sec. 5.1.2]{Rozenblum:2019}.

\begin{prop}
	Let $n\in \mathbb{N}^d$. The map $F\mapsto T^{\mathcal{F}^2_n}_F$ is injective from $L^\infty(\mathbb{C}^d)$ if and only if $n=0$. In other words, assigning a bounded function to a Toeplitz operator is only injective on the Bargmann-Fock space. 
\end{prop}

\section{A Tauberian theorem for bounded functions} %\label{sec:tauberian1}

As our first main result we present a generalization of Wiener's classical Tauberian theorem that applies to bounded functions and convolutions with integrable functions and trace class operators. The key tool is Werner's generalization of Wiener's approximation theorem from Theorem \ref{thm:wernerapproximation}.

\tauberianfunction

\begin{proof}
We start by proving that $(i)$ and $(ii)$ are equivalent. Assume $(i)$, and consider $a=S\star S\in L^1(\Rdd)$. Since $\F_\sigma (S\star S)(z)=\F_W(S)(z)^2$ for any $z\in \Rdd$ by \eqref{eq:fsigmaconv}, we obtain both that $\F_\sigma(a)$ has no zeros and (by evaluating the relation at $z=0$) that $$\int_{\Rdd} a(z) \ dz = \tr(S)\cdot \tr(S).$$ Then observe using associativity of the convolutions that
\begin{align*}
  f\ast a &= f\ast (S\star S) \\
  &= (f\star S)\star S \\
  &= (A\cdot\tr(S) \cdot I_{L^2} +K )\star S \\
  &= A\cdot \tr(S)\cdot\tr(S) +K\star S \quad \text{ by Lemma \ref{lem:convolutionswithidentity}} \\
  &= A\cdot \int_{\Rdd} a(z) \ dz +K\star S,
\end{align*}
and $K\star S\in C_0(\Rdd)$ by Lemma \ref{lem:compactL0}. The proof that $(ii)$ implies $(i)$ is similar by picking $S=a\star T$, where $T\in \tco$ is any operator in $\mathcal{W}$. Then $\F_W(S)(z)=\F_\sigma(a)(z)\F_W(T)(z)$ by \eqref{eq:fwignerconv}, so $\F_W(S)$ has no zeros and $\tr(S)=\int_{\Rdd} a(z) \ dz \cdot \tr(T)$ by evaluating the relation at $z=0$. Furthermore, associativity of convolutions gives
\begin{align*}
  f\star S &= f\star (a\star T) \\
  &= (f\ast a)\star T \\
  &= \left(A\cdot\int_{\Rdd}a(z) \ dz  +h \right)\star T \\
  &= A\cdot \int_{\Rdd}a(z) \ dz \cdot  \tr(T)\cdot I_{L^2} +h\star T \quad \text{ by Lemma \ref{lem:convolutionswithidentity}} \\
  &= A\cdot \tr(S)\cdot I_{L^2} +h\star T,
\end{align*}
and $h\star T\in \compacts$ by Lemma \ref{lem:compactL0}. Hence $(i)$ and $(ii)$ are equivalent.\\
The fact that $(ii)$ implies $(2)$ is Wiener's classical Tauberian theorem. The proof will therefore be completed if we can show $(i) \implies (1)$, so assume that $S$ satisfies $(i)$, and for now assume $A=0$. In short, we assume $f\star S\in \compacts$. We need to show that $f\star T\in \compacts$ for any $T\in \tco.$ Part $(3)$ of Theorem \ref{thm:wernerapproximation} implies that $T$ is the limit in the norm of $\tco$ of a sequence $r_n\star S$ for $r_n\in L^1(\Rdd)$. By commutativity and associativity of the convolutions, 
	\begin{equation*}
  f\star (r_n\star S)= r_n\star (f\star S)\in  \compacts \quad \text{ by Lemma \ref{lem:compactL0}.}
\end{equation*}
Proposition \ref{prop:young} then gives that
\begin{align*}
  \|f\star T-f\star (r_n\star S)\|_{\bo}
  &\leq \|f\|_{L^\infty} \|T-r_n\star S\|_{\tco}\to 0 \quad \text{ as } n\to\infty. 
\end{align*}
Hence $f\star T$ is the limit in the operator norm of compact operators, thus compact. Finally, assume that $A\neq 0$. Then $(f-A)\star S\in \compacts $ by Lemma \ref{lem:convolutionswithidentity}, so the result for $A=0$ implies that $(f-A)\star T\in \compacts$ for any $T\in \tco$, and applying Lemma \ref{lem:convolutionswithidentity} again we see that this is equivalent to $(1).$
\end{proof}

The case $A=0$ is particularly interesting, as it concerns the compactness of operators of the form $f\star T$ for $T\in \tco$. We will return to this special case on several occasions.

\begin{rem}
	\begin{enumerate}
		\item Note that the convolution of a bounded and an integrable function is continuous, so we lose no generality by assuming that $h$ and $h_g$ belong to $C_0(\Rdd)$ rather than merely assuming that they belong to $L^0(\Rdd).$
		\item As already mentioned in the proof, the classical Tauberian theorem of Wiener is the implication $(ii)\implies (2).$
		\item The conditions on the Fourier transforms of $S$ in $(i)$ are  necessary to imply $(1)$ and $(2)$. To see this, assume that $S\in \tco$ satisfies $\F_W(S)(z_0)=0$ for some $z_0=(x_0,\omega_0)\in \Rdd$. Then consider the function $f_{z_0}(z)=e^{2\pi i \sigma(z_0,z)}\in L^\infty(\Rdd)$. One can show that for any $T\in \tco$ we have
		\begin{equation*}
			f_{z_0} \star T = \F_W(T)(z_0) e^{-\pi i x_0\cdot \omega_0} \pi(z_0).
		\end{equation*}
		In particular, $f_{z_0}\star S=0\in \compacts$ since $\F_W(S)(z_0)=0$, so apart from the condition on $\F_W(S)$ we see that $S$ satisfies $(i)$ with $A=0$. However, $f_{z_0}\star T=\F_W(T)(z_0) e^{-\pi i x_0\cdot \omega_0} \pi(z_0)$ is not compact if  $\F_W(T)(z_0)\neq 0$, hence $(1)$ is not true for $f_{z_0}$. A similar argument with the same functions $f_{z_0}$ shows that the condition on $a$ in $(ii)$ is also necessary.
	\end{enumerate}
\end{rem}

\subsection{A result by Fern\'andez and Galbis}
In \cite{Fernandez:2006}, Fern\'andez and Galbis proved the following result on compactness of localization operators.

\begin{thm}[Fern\'andez and Galbis] \label{thm:fernandezgalbis}
	Let $f\in L^\infty(\Rdd)$. Then $\mathcal{A}^{\varphi_1,\varphi_2}_f$ is compact for all $\varphi_1,\varphi_2\in \mathscr{S}(\Rd)$ if and only if there is a non-zero $\Phi \in \mathscr{S}(\Rdd)$ such that for every $R>0$
\begin{equation} \label{eq:fernandezgalbis}
	\lim_{|x|\to \infty} \sup_{|\omega|\leq R} |V_\Phi f(x,\omega)| =0. 
\end{equation}

\end{thm}
\begin{rem}
\begin{enumerate}
	\item This requirement is weaker than both $f\in L^0(\Rdd)$ and $V_{\Phi}f\in C_0(\R^{4d})$ for some  non-zero $\Phi \in \mathscr{S}(\Rdd).$  Proving that either of these two statements implies compactness of $\mathcal{A}^{\varphi_1,\varphi_2}_f$ requires far less advanced tools than \eqref{eq:fernandezgalbis}, see \cite{Fernandez:2006}.
	\item The theorem holds for $f\in M^\infty(\Rdd)$, where $M^\infty(\Rdd)$ consists of all $f\in \mathscr{S}'(\Rdd)$ such that $V_{\varphi_0}f\in L^\infty(\R^{4d}).$ The space $M^\infty(\Rdd)$ contains $L^\infty(\Rdd)$ and certain distributions such as Dirac's delta distribution, see \cite{Grochenig:2001}.
\end{enumerate}
	
\end{rem}
This allows us to add another equivalent assumption to Theorem \ref{thm:tauberian1}, formulated in terms of the short-time Fourier transform of $f$.
\fernandezgalbis

\begin{proof}
	Consider the operator $S=\varphi_0 \otimes \varphi_0$. Then $S\in \mathcal{W}$ by \eqref{eq:stftgauss} and $f\star S=\mathcal{A}^{\varphi_0,\varphi_0}_f$ by \eqref{eq:locopsasconv}. If $(iii)$ is satisfied, Theorem \ref{thm:fernandezgalbis} implies using Lemma \ref{lem:convolutionswithidentity} that $$\mathcal{A}_{f-A}^{\varphi_0,\varphi_0}=(f-A)\star S=f\star S -A \cdot \tr(S) \cdot I_{L^2} $$ is compact, hence $(i)$ holds. If $(i)$ holds, then Theorem \ref{thm:tauberian1} $(1)$ implies that 
	$$f\star(\varphi_2\otimes \varphi_1)-A\cdot \tr(\varphi_2\otimes \varphi_1)\cdot I_{L^2}  = (f-A)\star (\varphi_2 \otimes \varphi_1)=\mathcal{A}^{\varphi_1,\varphi_2}_{f-A}$$
	 is compact for any $\varphi_1,\varphi_2 \in \mathscr{S}(\Rd)$, so Theorem \ref{thm:fernandezgalbis} implies that $(iii)$ holds. 
\end{proof}

\begin{rem}
	One may easily calculate that $$V_{\Phi}(f-A)(x,\omega)=V_\Phi f(x,\omega)-A\cdot e^{-2\pi i x \cdot \omega} \widehat{F}(\overline{\Phi})(\omega).$$ Condition $(iii)$ therefore says that for any $R>0$, if fixed $x$ is picked with $|x|$ sufficiently large, then $V_\Phi f(x,\omega)$ should uniformly approximate $A\cdot e^{-2\pi i x \cdot \omega} \widehat{F}(\overline{\Phi})(\omega)$ for $|\omega|\leq R.$
\end{rem}

Theorem \ref{thm:fernandezgalbis} is a theorem concerning compactness of operators -- its proof in \cite{Fernandez:2006} relies on results on relatively compact subsets of $\compacts$. However, Theorem \ref{thm:tauberian1} along with Proposition \ref{prop:fernandezgalbis} allows us to translate the result to functions on $\Rdd$. In fact, it leads to a characterization in terms of the short-time Fourier transform of those $f\in L^\infty(\Rdd)$ satisfying the assumptions of Wiener's classical Tauberian theorem. To our knowledge this result is new, so we formulate it as a separate statement.

\begin{thm}
	Let $A\in \mathbb{C}$ and $f\in L^\infty(\Rdd)$ be given. The following are equivalent.
	\begin{itemize}
		\item There is some non-zero $\Phi \in \mathscr{S}(\Rdd)$ such that for every $R>0$ 
	 \begin{equation*} 
	\lim_{|x|\to \infty} \sup_{|\omega|\leq R} |V_\Phi (f-A)(x,\omega)| =0. 
\end{equation*}
		\item There is $a\in W(\Rdd)$ and $h\in C_0(\Rdd)$ such that $$f\ast a = A\cdot \int_{\Rdd} a(z) \ dz + h.$$ 
		\item For any $g \in L^1(\Rdd)$ there is $h_g \in C_0(\Rdd)$ such that $$f\ast g=A\cdot \int_{\Rdd} g(z) \ dz + h_g.$$
	\end{itemize}
\end{thm}

\begin{rem}
	One might naturally ask if this result holds for $\Rd$ for any $d\geq 1$, and not just for even $d$. Our proof exploits Theorem \ref{thm:fernandezgalbis}, which has no analogue for odd $d$. We can therefore not extend  the proof to the general case.  
\end{rem}

\subsection{A closer look at the two assumptions of Theorem \ref{thm:tauberian1}}

By Remark \ref{rem:sufficientcompact} and Lemma \ref{lem:convolutionswithidentity}, $f\in L^\infty(\Rdd)$ trivially satisfies the assumptions (and conclusions) in Theorem \ref{thm:tauberian1} if $f=A + h$ for some $A\in \mathbb{C}$ and $h\in L^p(\Rdd)$ for $1\leq p <\infty$ or $p=0$. We will now see examples that do not satisfy these conditions. 

\begin{exmp} \label{exmp:tauberian1}
	\begin{enumerate}
		\item In \cite[Prop. 4.1]{Fernandez:2007}, Galbis and Fern\'andez show that the function  $f(x,\omega)=e^{i\pi |z|^2}$ satisfies condition $(iii)$ from Proposition \ref{prop:fernandezgalbis}, hence it satisfies $(i)$ and $(ii)$ in Theorem \ref{thm:tauberian1}. Clearly $f\notin L^p(\Rdd)$ for $p=0$ or $1\leq p <\infty$. 
		\item Given $\tau \in (0,1)\setminus \{1/2\}$, the function $
	a_{\tau}(x,\omega)=
		 \frac{2^d}{|2\tau-1|^d}\cdot e^{2\pi i \frac{2}{2\tau -1}x\cdot \omega} $ satisfies assumption $(i)$ in Theorem \ref{thm:tauberian1} with $A=0$, as we prove in Proposition \ref{prop:tauwienerexample}. Again, we see that $a_\tau \notin L^p(\Rdd)$ for $p=0$ or $1\leq p <\infty$.
%		\item Since assumption $(ii)$ in Theorem \ref{thm:tauberian1} is the assumption in Wiener's classical Tauberian theorem, we can use examples constructed in classical Tauberian theory. For instance, \cite[Thm. Ch. II Prop. 16.1]{Korevaar:2004} constructs, for any $a\in W(\R)$ , a real-valued, continuous $f\in L^\infty(\R)$ with $-1\leq f(x) \leq 1$ such that $|f(x)|$ has no limit as $|x|\to \infty$, yet $f\ast a\in C_0(\R).$ In fact, $f$ is constructed as an infinite sum of non-overlapping "zigzag" functions. It is not difficult to extend this construction to $\Rdd$ by taking tensor products and picking $a=\varphi_0\in W(\Rdd)$, hence obtaining examples of continuous, real-valued $f\in L^\infty(\Rdd)$ where $|f(z)|$ has no limit as $|z|\to \infty$ yet assumption $(ii)$ is still satisfied for $A=0$.
		\item If $f\in L^\infty(\Rdd)$ is a so-called \textit{pseudomeasure}, meaning that $\F_\sigma(f)\in L^\infty(\Rdd)$, then $f$ satisfies $(ii)$ with $A=0$. To see this, let $a(z)=e^{-\pi  |z|^2}$. Then $\F_\sigma(a)=a$ has no zeros, and $$f\ast a=\F_\sigma \F_\sigma (f\ast a)=\F_\sigma(\F_\sigma(f)\cdot a),$$ and since $\F_\sigma(f)\in L^\infty(\Rdd)$ we have $\F_\sigma(f)\cdot a \in L^1(\Rdd)$. Hence $f\ast a\in C_0(\Rdd)$ by the Riemann-Lebesgue lemma.  \\
		Rather surprisingly, we may prove $(1)$ directly in this case by considering the operator side of our setup. For any $T\in \tco$, we obtain that $\F_W(T)\in L^2(\Rdd)$ since $\tco \subset \HS$ and $\F_W:\HS \to L^2(\Rdd)$ is a unitary operator. By our assumption on $f$, it follows that $\F_W(f\star T)=\F_\sigma(f)\F_W(T)\in L^2(\Rdd)$, hence $f\star T \in \HS\subset \compacts$. The key to this calculation is the inclusion $L^\infty(\Rdd)\cdot \F_W(\tco) \subset L^2(\Rdd)$ -- the corresponding function result $L^\infty\cdot \F_\sigma(L^1)\subset L^2$ is not true by the results in \cite{Braun:1983}. 
	\end{enumerate}
\end{exmp}

%\begin{cor}
%	There is a real-valued, continuous $f\in L^\infty(\Rdd)$ such that $f\star S$ is compact for any $S\in \tco$, yet $\lim_{|z|\to \infty} |f(z)|$ does not exist. In particular, $\mathcal{A}^{\varphi_1,\varphi_2}_f$ is compact for all $\varphi_1,\varphi_2\in L^2(\Rd).$ 
%\end{cor}

The examples above show that it is not necessary to have $\lim_{|z|\to \infty} f(z) = 0$ in order to satisfy assumptions $(i)$ and $(ii)$ with $A=0$. A well-known result in the Tauberian theory of functions due to Pitt \cite{Pitt:1938} says that if we assume that $f$ is \textit{slowly oscillating}, then $\lim_{|z|\to \infty} f(z) = 0$ is necessary for $f$ to satisfy $(ii)$. 

Recall that $f$ is slowly oscillating on $\Rdd$ if for every $\epsilon >0$ there is $\delta>0$ and $K>0$ such that $|f(z)-f(z-z')|<\epsilon$ for $|z'|<\delta$ and $|z|>K$. We refer to \cite[Thm. 4.74]{Folland:2015} for a formulation of Pitt's result that applies to $\Rdd$.

\begin{thm}[Pitt] \label{thm:pitt}
	If $f\in L^\infty(\Rdd)$ is slowly oscillating and satisfies either assumption $(i)$ or $(ii)$ in Theorem \ref{thm:tauberian1} or $(iii)$ from Proposition \ref{prop:fernandezgalbis} with $A=0$, then $f\in L^0(\Rdd).$
\end{thm}

\begin{rem}
	Any uniformly continuous $f\in L^\infty(\Rdd)$ is slowly oscillating, hence if such $f$ satisfies $(i)$ with $A=0$, then $f\in C_0(\Rdd)$. This weaker statement actually follows from the correspondence theory introduced by Werner in \cite{Werner:1984}, more precisely by \cite[Thm. 4.1 (3)]{Werner:1984}. In Werner's terminology $C_0(\Rdd)$ and $\compacts$ are corresponding subspaces, since convolutions with trace class operators interchanges these two spaces by Lemma \ref{lem:compactL0}. We will see the operator-analogue of this result in Section \ref{sec:pittoperator}
\end{rem}
\subsubsection{Consequences for Toeplitz operators}

We now formulate a version of the Tauberian theorem for (polyanalytic) Bargmann-Fock Toeplitz operators. As a preliminary observation, let $\mathcal{H}_1,\mathcal{H}_2$ be two Hilbert spaces. If $S\in \mathcal{L}(\mathcal{H}_1)$ and $T\in \mathcal{L}(\mathcal{H}_2)$ are unitarily equivalent, i.e. there is unitary $U:\mathcal{H}_1\to \mathcal{H}_2$ such that $S=U^*TU$, then one easily checks that $S=A\cdot I_{\mathcal{H}_1}+K_1$ for $A\in \mathbb{C}$ and compact $K_1\in \mathcal{L}(\mathcal{H}_1)$ if and only if $T=A\cdot I_{\mathcal{H}_2}+K_2$ for compact $K_2\in \mathcal{L}(\mathcal{H}_2)$.

\begin{prop}
Let $F\in L^\infty(\mathbb{C}^d)$ and $A\in \mathbb{C}$. Define $f\in L^\infty(\Rdd)$ by $f(x,\omega)=F(x-i\omega)$. The following are equivalent:
\begin{enumerate}[(i)]
	\item $T^{\mathcal{F}^2}_F=A\cdot I_{\mathcal{F}^2} + \tilde{K}_0$ for some compact operator $\tilde{K}_0$ on $\mathcal{F}^2(\mathbb{C}^d)$.
	\item There is some $a\in W(\Rdd)$ such that $$f\ast a = A\cdot \int_{\Rdd} a(z) \ dz + h_a$$ for some $h_a\in C_0(\Rdd)$.
	\item There is some non-zero $\Phi \in \mathscr{S}(\Rdd)$ such that for every $R>0$
	 \begin{equation*} 
	\lim_{|x|\to \infty} \sup_{|\omega|\leq R} |V_\Phi (f-A)(x,\omega)| =0. 
\end{equation*}
\end{enumerate}	
Furthermore, if any of the equivalent conditions above holds, then for any $n\in \mathbb{N}^d$ the polyanalytic Toeplitz operator $T_{F}^{\mathcal{F}^2_n}$ is of the form $$T_{F}^{\mathcal{F}^2_n}=A\cdot I_{\mathcal{F}^2_n}+\tilde{K}_n,$$ where $\tilde{K}_n$ is a compact operator on $\mathcal{F}^2_n(\mathbb{C}^d)$.
\end{prop}

\begin{proof}
	By Proposition \ref{prop:unitarilyequivalent}, $T^{\mathcal{F}^2}_F$ is unitarily equivalent to $\mathcal{A}^{\varphi_0,\varphi_0}_f=f\star (\varphi_0 \otimes \varphi_0).$ By the remark above, part $(i)$ holds if and only if $f\star (\varphi_0 \otimes \varphi_0)=A\cdot I_{L^2} +K_0$ for some compact operator $K_0$ on $L^2(\Rd)$. Since $\varphi_0\otimes \varphi_0\in \mathcal{W}$ by \eqref{eq:stftgauss}, the fact that $(i),(ii)$ and $(iii)$ are equivalent follows from Proposition \ref{prop:fernandezgalbis}.
	
	 As we have seen that $(i)$ implies that $f\star (\varphi_0 \otimes \varphi_0)=A\cdot I_{L^2} +K_0$ and that $\varphi_0 \otimes \varphi_0\in \mathcal{W}$, Theorem \ref{thm:tauberian1} implies that for every $n$ there is a compact $K_n$ with $$f\star (\varphi_n \otimes \varphi_n)=A\cdot I_{L^2} \cdot  \tr(\varphi_n \otimes \varphi_n)+K_n=A\cdot I_{L^2}  +K_n.$$ The last statement then follows as $T_{F}^{\mathcal{F}^2_n}$ is unitarily equivalent to $\mathcal{A}^{\varphi_n,\varphi_n}_f=f\star (\varphi_n \otimes \varphi_n)$ by Proposition \ref{prop:uneqpolyana}.
\end{proof}
\begin{rem}
 The equivalence of $(i)$ and $(ii)$ when $a$ is fixed to be the Gaussian $a(x,\omega)=e^{-\pi(x^2+\omega^2)}$ is due to Engli\v{s}, see the equivalence of $(a)$ and $(c)$ in \cite[Thm. B]{Englis:1999}. Note that Engli\v{s} also considers products of Toeplitz operators, which is a setting we will return to in Section \ref{sec:pittoperator}.
\end{rem}
The same reasoning gives the following Tauberian theorem for Toeplitz operators on Gabor spaces using Proposition \ref{prop:locoptoeplitz}.

\begin{prop}
	Let $f\in L^\infty(\Rdd)$ and $A\in \mathbb{C}$. The following are equivalent.
	\begin{enumerate}[(i)]
		\item There is some $\varphi \in L^2(\Rd)$ such that $V_\varphi\varphi$ has no zeros and $T^\varphi_f=A\cdot I_{V_\varphi(L^2)}+K$ for some compact operator $K\in \mathcal{L}(V_\varphi(L^2))$ .
		\item There is some $a\in W(\Rdd)$ such that
	$
  f\ast a=A\cdot \int_{\Rdd} a (z) \ dz+h 
$
for some $h\in C_0(\Rdd)$.
		\item There is some non-zero $\Phi \in \mathscr{S}(\Rdd)$ such that for every $R>0$ 
	$$
	\lim_{|x|\to \infty} \sup_{|\omega|\leq R} |V_\Phi (f-A)(x,\omega)| =0. 
$$
	\end{enumerate}
	Furthermore, if any of the equivalent conditions above holds, then for every normalized $\varphi'\in L^2(\Rd)$ we have that $T^{\varphi'}_f$ is of the form $A\cdot I_{V_{\varphi'}(L^2)}+K_{\varphi '}$ for some compact operator $K_{\varphi '}\in \mathcal{L}(V_{\varphi '}(L^2))$.
	\end{prop}

\subsection{Injectivity of localization operators and Riesz theory of compact operators}

We will now let the kind of operators appearing in Theorem \ref{thm:tauberian1} inspire a slight detour that does not explicitly build on the Tauberian theorems. Theorem \ref{thm:tauberian1} gives conditions to ensure that a localization operator $\mathcal{A}^{\varphi,\varphi}_f$ is a compact perturbation of a scaling of the identity, i.e. of the form $A\cdot I_{L^2}  +K$ for $0\neq A \in \mathbb{C}$ and $K\in \compacts$. The theory of such operators, sometimes referred to as Riesz theory due to the seminal work of F. Riesz \cite{Riesz:1916}, contains several powerful results similar to those that hold for matrices. We will use the following result, see \cite[Lem. 6.30 \& Thm. 6.33]{Bowers:2014} for proofs.

\begin{prop} \label{prop:riesztheory}
	Assume that $T\in \bo$ is of the form $A\cdot I_{L^2} +K$ for $A\neq 0$ and $K\in \compacts$. Then $T$ has closed range and $\dim(\ker T)=\dim (\mathrm{coker}(T))<\infty$. In particular, $T$ is injective if and only if $T$ is surjective.
\end{prop}
As an obvious consequence, we note that if $\mathcal{A}_{f}^{\varphi,\varphi}=A\cdot I_{L^2}+K$ for $A\neq 0$ and $K\in \compacts$ and $\mathcal{A}^{\varphi,\varphi}_f$ is injective, then $\mathcal{A}^{\varphi,\varphi}_f$ is an isomorphism on $L^2(\Rd)$. Inspired by this, we investigate conditions ensuring that localization operators are injective. The proof of the next result is similar to that of \cite[Lem. 1.4]{Boggiatto:2005}.

\begin{lem}  \label{lem:injectivelocop}
	Assume that $f\in L^\infty(\Rdd)$ such that $f(z)\geq 0$ for a.e. $z\in \Rdd$. 
	
	\begin{enumerate}
	\item 	If $0\neq \varphi\in L^2(\Rdd)$ and there is $\Delta\subset \Rdd$ of finite Lebesgue measure with $$f(z)> 0 \quad \text{ for a.e. } z\notin \Delta,$$
	then the localization operator $\mathcal{A}^{\varphi,\varphi}_f$ is injective. 
	\item If there is some open subset $\Omega \subset \Rdd$ such that $$f(z)> 0 \quad \text{ for a.e. } z\in \Omega,$$ then the localization operator $\mathcal{A}^{\varphi_0,\varphi_0}_f$ is injective.
	\end{enumerate}
\end{lem}

\begin{proof}
	We first prove $(1).$ Assume that $\mathcal{A}^{\varphi,\varphi}_f(\psi)=0$. This implies by \eqref{eq:weaklocop} that 
	\begin{equation*}
  \inner{\mathcal{A}^{\varphi,\varphi}_f (\psi)}{\psi}_{L^2}=\int_{\Rdd} f(z) |V_\varphi \psi(z)|^2 \ dz =0.
\end{equation*}
Since we assume that $f$ is non-negative for a.e. $z$, this further implies that 
\begin{equation*}% \label{eq:injectivelocop1}
  \int_{\Rdd\setminus \Delta} f(z) |V_\varphi \psi(z)|^2 \ dz=0.
\end{equation*}
This implies that $V_\varphi \psi(z)=0$ for a.e. $z\notin \Delta$. Hence $\psi=0$, as the main result of \cite{Janssen:1998a} says that $V_\varphi \psi(z)$ cannot be supported on a set of finite Lebesgue measure unless $\psi=0$ or $\varphi=0$. 

To prove $(2)$, a similar argument as above shows that $\mathcal{A}^{\varphi_0,\varphi_0}_f(\psi)=0$ implies that $V_{\varphi_0}\psi(z)=0$ for a.e. $z\in \Omega$. Continuity gives that $V_{\varphi_0}\psi(z)=0$ for all $z\in \Omega$. The analytic function $\mathcal{B}(\psi)(x+i\omega)=e^{-\pi i x \cdot \omega}e^{-\frac{\pi}{2}(x^2+\omega^2)} V_{\varphi_0}\psi(x,-\omega)$ therefore vanishes on an open subset of $\mathbb{C}^d$, hence $\mathcal{B}(\psi)=0$ by uniqueness of analytic continuation. Thus $\psi=0$ as $\mathcal{B}$ is injective.
\end{proof}
We deduce sufficient conditions for localization operators to be isomorphisms.

\localizationisomorphism
\begin{proof}
	By Lemma \ref{lem:injectivelocop} part $(1)$, $\mathcal{A}^{\varphi,\varphi}_f$ is injective. By assumption $(iii)$, Theorem \ref{thm:tauberian1} gives that $a\star (\varphi\otimes \varphi)\in \compacts$, so that $$\mathcal{A}^{\varphi,\varphi}_f=(M+a)\star (\varphi\otimes \varphi)=M\cdot \|\varphi\|_{L^2}^2 \cdot I_{L^2}+a\star (\varphi\otimes \varphi)$$ is a compact perturbation of a scaling of the identity. Hence Proposition \ref{prop:riesztheory} implies that $\mathcal{A}^{\varphi,\varphi}_f$ is also surjective.
\end{proof}

\begin{rem} \label{rem:iso}
\begin{enumerate}
	\item Finding specific examples of $a$ satisfying the assumptions above is not difficult, but it is worth noting that $a$ need not vanish at infinity. For instance, a standard construction gives continuous $a\in L^1(\Rdd)\cap L^\infty(\Rdd)$ such that $0\leq a \leq 1$, $\limsup_{|z|\to \infty} |a(z)|=1$ and $\liminf_{|z|\to \infty} |a(z)|=0$. Then $a$ satisfies all three conditions above for $M>0$, even though $f=M+a$ has no limit as $|z|\to \infty$.  Of course, if we add the condition that $a$ is slowly oscillating, then $a$ must vanish at infinity by Theorem \ref{thm:pitt}. 
	\item Other isomorphism theorems for localization operators may be found in \cite{Boggiatto:2005,Grochenig:2011toft,Grochenig:2013}.
\end{enumerate}
\end{rem}
We state a special case of Proposition \ref{prop:isomorphisms} as a theorem, namely the case where $f=\chi_{\Omega}$  such that $\Omega^c$ has finite measure. We find that as long as $\Omega^c$ has finite measure, the values of $V_\varphi \psi(z)$ for $z\in \Omega^c$ are not needed to reconstruct $\psi$ -- independently of the geometry of $\Omega$ and the window $\varphi$.
\begin{thm} \label{thm:iso}
	Assume that $\Omega\subset \Rdd$ satisfies that $\Omega^c$ has finite Lebesgue measure, and that $0\neq \varphi \in L^2(\Rd)$. Then the localization operator $\mathcal{A}^{\varphi,\varphi}_{\chi_\Omega}$ is an isomorphism on $L^2(\Rd)$. 
	In particular, any $0\neq \psi \in L^2(\Rd)$ is uniquely determined by the values of $V_\varphi \psi(z)$ for $z\in \Omega$ and there exist constants $C,D>0$ independent of $\psi$ such that $$C\cdot \|\psi\|_{L^2}\leq   \left\| \int_{\Omega} V_\varphi \psi(z) \pi(z) \varphi \ dz \right\|_{L^2}\leq D\cdot \|\psi\|_{L^2}. $$ 
\end{thm}

\begin{proof}
This is a special case of Proposition \ref{prop:isomorphisms} with $M=1$ and $a=-\chi_{\Omega^c}$. Then $f=1-\chi_{\Omega^c}=\chi_{\Omega}$, and one easily checks that the conditions in the proposition are satisfied with $\Delta=\Omega^c$, in particular $(iii)$ follows as $\chi_{\Omega^c}\in L^1(\Rdd).$ 
\end{proof}

\begin{rem}
\begin{enumerate}
   \item After submitting this paper for publication, we were made aware that stronger versions of this result by different approaches exist in the literature, see \cite{Fernandez:2010} and references therein. To our knowledge the strongest of these results is \cite{Fernandez:2010}, where it is shown that Theorem \ref{thm:iso} holds if $\chi_{\Omega^c}$ satisfies assumption (i) or (ii) of Theorem \ref{thm:tauberian1}. It follows that this assumption is sufficient in part (1) of Proposition \ref{prop:BFisomorphism} as well.  
    \item Theorem \ref{thm:iso} is an example of turning uncertainty principles into signal recovery results, as proposed by Donoho and Stark \cite{Donoho:1989}. An alternative proof more in line with the methods of \cite{Donoho:1989} could be obtained by showing that $\|\mathcal{A}^{\varphi,\varphi}_{\chi_{\Omega^c}}\|_{\bo}<1$ using \cite{Janssen:1998a}, and using a Neumann series argument to deduce the invertibility of $\mathcal{A}^{\varphi,\varphi}_{\chi_{\Omega}}=I_{L
   ^2}-\mathcal{A}^{\varphi,\varphi}_{\chi_{\Omega^c}}$.
    \item If $\varphi$ belongs to \textit{Feichtinger's algebra} $M^1(\Rd)$\cite{Feichtinger:1981,Grochenig:2001}, then invertibility of $\mathcal{A}^{\varphi,\varphi}_f$ on $L^2(\Rd)$ implies that $\mathcal{A}^{\varphi,\varphi}_f$ is also invertible on all \textit{modulation spaces} $M^{p,q}(\Rd)$ for $1\leq p,q \leq \infty$ (see \cite{Grochenig:2001} for an introduction to modulation spaces). This follows by combining \cite[Thm. 3.2]{Cordero:2003} and \cite[Cor. 4.7]{Grochenig:2006}.
\end{enumerate}
	
\end{rem}
\subsubsection{Isomorphism results for $\mathcal{F}^2_n(\C^d)$}
Any Toeplitz operator $T^{\mathcal{F}^2_n}_F$ on polyanalytic Bargmann-Fock space is unitarily equivalent to a localization operator $\mathcal{A}^{\varphi_n,\varphi_n}_f$ by Proposition \ref{prop:uneqpolyana}, where $f\in L^\infty(\Rdd)$ and $F\in L^\infty(\C^d)$ are related by $$F(x+i\omega)=f(x,-\omega).$$ Hence the results of this section may be translated into results for Toeplitz operators on $\mathcal{F}_n^{2}(\C^d).$ We include a couple of such results in the next statement. One may of course obtain isomorphism results for Gabor spaces in the same way by using Proposition \ref{prop:locoptoeplitz}.

\BFisomorphism
\begin{proof}
	In light of Proposition \ref{prop:uneqpolyana}, the first part follows from Theorem \ref{thm:iso} and the second from Remark \ref{rem:iso}.
\end{proof}

\section{A Tauberian theorem for bounded operators} %\label{sec:tauberian2}

A guiding principle in the theory of quantum harmonic analysis is that the role of functions and operators may often be interchanged in theorems. It should therefore come as no surprise that we can prove a Tauberian theorem where the bounded function $f$ from Theorem \ref{thm:tauberian1} is replaced by a bounded operator $R$, with just a few modifications of the proof.
\tauberianoperator

\begin{proof}
	The equivalence of the assumptions is proved in a similar way as for Theorem \ref{thm:tauberian1}: for $(i)\implies (ii)$ pick $a=S\star S$, and for $(ii)\implies (i)$ pick $S=a\star T$ for any $T\in \mathcal{W}$. 
	
Then assume that $(i)$ holds with $A=0$, the extension to $A\neq 0$ is done as in the proof of Theorem \ref{thm:tauberian1}. To show $(1)$, one proceeds as in the proof of Theorem \ref{thm:tauberian1} by first showing that $S\star T\in C_0(\Rdd)$ if $T=r\star S$ for some $r\in L^1(\Rdd)$. Using Theorem \ref{thm:wernerapproximation} one has that any $T\in \tco$ is the limit in the norm of $\tco$ of a sequence $r_n\star S$ with $r_n\in L^1(\Rdd)$. The proof is completed by showing that the sequence $R\star (r_n\star S)$ -- which is a sequence of functions in $C_0(\Rdd)$ -- converges uniformly to $R\star T$. Since $C_0(\Rdd)$ is closed under uniform limits, this implies $(1)$.

The proof that $(i)$ implies $(2)$ follows the same pattern. First show it for $g=T\star S$ for some $T\in \tco$, then extend to all $g$ by density, since Theorem \ref{thm:wernerapproximation} implies that any $g\in L ^1(\Rdd)$ is the limit of a sequence $T_n\star S$ for $T_n\in \tco.$
\end{proof}

\begin{rem} %\label{rem:tauberian2nec}
	The conditions on the Fourier transforms of $S$ and $a$ in $(i)$ and $(ii)$ are necessary to imply $(1)$ and $(2)$, as can be shown by picking $R=\pi(z_0)$ for $z_0=(x_0,\omega_0)\in \Rdd$. A calculation from the definitions \eqref{eq:defconvopop} and \eqref{eq:deffourierwigner} shows that $$\left[\pi(z_0)\star S\right](z)=e^{2\pi i \sigma(z_0,z)}e^{\pi i x_0 \cdot \omega_0}\F_W(S)(z_0).$$ So if $\F_W(S)(z_0)=0$, we get that $\pi(z_0)\star S=0\in C_0(\Rdd)$. On the other hand we may consider $\varphi_0 \otimes \varphi_0$. By Example \ref{exmp:gaussianfw}, we get that $$\left[\pi(z_0)\star (\varphi_0 \otimes \varphi)\right](z)=e^{2\pi i \sigma(z_0,z)}e^{\pi i x_0 \cdot \omega_0}e^{-\pi z_0^2}\notin C_0(\Rdd).$$ Hence the condition in $(i)$ is necessary. To show that the condition on $a$ in $(ii)$ is necessary one uses a similar argument and the fact that $$\pi(z_0)\star a = \F_\sigma(a)(z_0) \pi(z_0),$$ as a calculation shows. 
\end{rem}

From Lemmas \ref{lem:compactL0} and \ref{lem:convolutionswithidentity} it is clear that $(i)$ and $(ii)$ are satisfied if $R=A\cdot I_{L^2}  + K$ for some compact operator $K$. However, these are not the only examples.

\begin{exmp} \label{exmp:operatorpseudomeasures}
 If $R\in \bo$ satisfies that $\F_W(R)\in L^\infty(\Rdd)$, then $R$ satisfies assumption $(ii)$ of Theorem \ref{thm:tauberian2} with $A=0$ -- such $R$ are the operator-analogues of the pseudomeasures considered in Example \ref{exmp:tauberian1}. To prove this, let $S=\varphi_0 \otimes \varphi_0$. Then $\F_W(S)(z)=e^{-\pi |z|^2}$, so $S\in \mathcal{W}$, and $$\F_\sigma(R\star S)=\F_W(R)\cdot \F_W(S)\in L^1(\Rdd).$$ By Fourier inversion we have $$R\star S=\F_\sigma(\F_W(R)\cdot \F_W(S)),$$ 
			which belongs to $C_0(\Rdd)$ by the Riemann-Lebesgue lemma. 
			
An example of such $R$ is $R=P$, the parity operator. One can show that $\F_W(P)(z)=2^d$ for any $z\in \Rdd$, hence $P$ is a non-compact operator satisfying assumption $(ii)$ of Theorem \ref{thm:tauberian2} with $A=0$. We will return to this and other examples below. 
\end{exmp}

\subsection{Pitt improvements and characterizing compactness using Berezin transforms}  \label{sec:pittoperator}

As we saw in Theorem \ref{thm:pitt}, Pitt's classical theorem gives a condition on $f\in L^\infty(\Rdd)$ that ensures that $$f\ast g\in C_0(\Rdd)\text{ for } g\in W(\Rdd) \implies f\in C_0(\Rdd).$$

In particular, we noted that this is true if $f$ is uniformly continuous. To generalize this statement to operators $R\in \bo$, recall that $f\in L^\infty(\Rdd)$ is uniformly continuous if and only if $z\mapsto T_z(f)$ is continuous map from $\Rdd$ to $L^\infty(\Rdd)$. Hence a natural analogue of the uniformly continuous functions is the set $$\mathcal{C}_1 :=\{R\in \bo: z\mapsto \alpha_z(R) \text{ is continuous from } \Rdd \text{ to } \bo\};$$ this heuristic was also followed by Werner \cite{Werner:1984} and Bekka \cite{Bekka:1990}. With this in mind, the following result from \cite{Werner:1984} is an analogue of Pitt's theorem for operators.

\wernerpitt
\begin{proof}
	That the first statement implies the other two is Lemma \ref{lem:compactL0}. That the other statements imply the first follows from the theory of \textit{corresponding subspaces} developed by Werner in \cite{Werner:1984}, more precisely from \cite[Thm. 4.1 (3)]{Werner:1984}. In the notation of \cite{Werner:1984} we have picked $\mathcal{D}_0=C_0(\Rdd)$ and $\mathcal{D}_1=\compacts$.
\end{proof}

We then try to gain a better understanding of the elements of $\mathcal{C}_1.$
\begin{lem} The following set inclusion and equality hold:
	\begin{equation} \label{eq:toeplitzareuc}
	L^\infty(\Rdd)\star \tco \subset L^1(\Rdd)\star \bo = \mathcal{C}_1.
\end{equation}
\end{lem}
\begin{proof}
	The equality $\mathcal{C}_1=L^1(\Rdd)\star \bo$ is \cite[Prop. 4.5]{Luef:2018c}. Then assume $R=f\star S$ for $f\in L^\infty(\Rdd)$ and $S\in \tco$. By \cite[Prop. 7.4]{Luef:2018c} there must exist $g\in L^1(\Rdd)$ and $T\in \tco$ such that $S=g\star T.$ It follows by associativity and commutativity of convolutions that we have $R=f\star (g\star T)=g\star (f\star T).$ Since $f\star T\in \bo$ by Proposition \ref{prop:young}, it follows that $R\in L^1(\Rdd)\star \bo$. 
\end{proof}

Furthermore, it is not difficult to see that $\mathcal{C}_1$ equipped with the operator norm is a Banach algebra. Hence it must contain the Banach algebra generated by elements of the form $f\star T$ for $f\in L^\infty(\Rdd)$ and $T\in \tco,$ and Proposition \ref{thm:wernerpitt} applies to operators in this Banach algebra.

This allows us to apply the results above to characterizing compactness of Toeplitz operators by their Berezin transform, a much-studied question going back to results of Axler and Zheng \cite{Axler:1998} for the so-called Bergman space, and soon after Engli\v{s} \cite{Englis:1999} for the Bargmann-Fock space $\mathcal{F}^2(\C^d)$. The central question is whether a Toeplitz operator on a reproducing kernel Hilbert space must be compact if its Berezin transform vanishes at infinity -- see Section 4 of \cite{Bauer:2019} for an overview over results of this nature in the literature. We will use Proposition \ref{thm:wernerpitt} to reprove the main result of \cite{Bauer:2012} for $\mathcal{F}^2(\C^d)$ and extend it to a class of Gabor spaces, but we hasten to add that the method of proving the results of \cite{Bauer:2012} using the results of \cite{Werner:1984} was already noted recently by Fulsche \cite{Fulsche:2019}.
Before the proof, recall the linear and multiplicative isometric isomorphism $\Theta^\varphi:\mathcal{L}(V_\varphi(L^2)) \to \bo$ from \eqref{eq:theta}, which satisfies that $\Theta^\varphi(T^\varphi_f)=\mathcal{A}^{\varphi,\varphi}_{f}$ and $\mathfrak{B}^\varphi \tilde{T}=\Theta^\varphi(\tilde{T})\star (\check{\varphi}\otimes \check{\varphi})$.

\axlerzheng

\begin{proof}
First note that the assumption on $V_\varphi \varphi$ means that $\varphi\otimes \varphi \in \mathcal{W}$ by Lemma \ref{lem:rankonecase}, and as a simple calculation shows that $\F_W(\check{\varphi}\otimes \check{\varphi})(z)=\F_W(\varphi\otimes \varphi)(-z)$ it also means that $\check{\varphi}\otimes \check{\varphi} \in \mathcal{W}$.
 To see that the first statement implies the second, note that $\Theta^\varphi(\tilde{T})$ is compact if and only if $\tilde{T}$ is, so $$\mathfrak{B}^\varphi \tilde{T} = \Theta^\varphi(\tilde{T})\star (\check{\varphi}\otimes \check{\varphi})\in C_0(\Rdd)$$ by Lemma \ref{lem:compactL0}. For the other direction, it is clear by the properties of $\Theta^\varphi$ that it maps $\mathcal{T}^\varphi$ into the Banach algebra generated by localization operators $\mathcal{A}^{\varphi,\varphi}_f=f\star(\varphi\otimes \varphi)$ for $f\in L^\infty(\Rdd)$. In particular, $\Theta^\varphi(\mathcal{T}^\varphi)\subset \mathcal{C}_1$ by \eqref{eq:toeplitzareuc} as $\mathcal{C}_1$ is a Banach algebra containing $\mathcal{A}^{\varphi,\varphi}_f$ for all $f\in L^\infty(\Rdd)$. Since $\mathfrak{B}^\varphi \tilde{T} = \Theta^\varphi(\tilde{T})\star (\check{\varphi}\otimes \check{\varphi})\in C_0(\Rdd)$ and $\check{\varphi}\otimes \check{\varphi}\in \mathcal{W}$ by assumption, Proposition \ref{thm:wernerpitt} gives that $\Theta^\varphi(\tilde{T})$ is compact, hence $\tilde{T}$ is compact as $\Theta^\varphi$ is a unitary equivalence by definition. 
 
 The last statement follows from Theorem \ref{thm:pitt}, as $T_{f}^\varphi$ is compact if and only if $\Theta^\varphi(T^\varphi_f)=\mathcal{A}^{\varphi,\varphi}_{f}=f\star (\varphi\otimes \varphi)$ is compact, and $\varphi\otimes \varphi\in \mathcal{W}.$
\end{proof}
\begin{rem}
Similar techniques have also recently been used by Hagger \cite{Hagger:2020} to give a characterization of some generalizations of $\mathcal{T}^\varphi$. 
\end{rem}
There are several examples of $\varphi$ satisfying that $V_\varphi \varphi$ has no zeros, which by the proposition gives examples of reproducing kernel Hilbert spaces $V_\varphi (L^2)$ such that Toeplitz operators are compact if and only if their Berezin transform vanishes at infinity. One example is the one-sided exponential $\varphi(t)=\chi_{[0,\infty)}(t) e^{-t}$ for $t\in \R$ considered by Janssen \cite{Janssen:1996}, and new examples were recently explored in \cite{Grochenig:2019}.

Essentially the same argument as for Theorem \ref{thm:axlerzheng}, only replacing $\Theta^\varphi$ by the map $\Theta^{\mathcal{F}^2}:\mathcal{L}(\mathcal{F}^2(\C^d))\to \bo$ defined by $\Theta^{\mathcal{F}^2}(\tilde{T})=\mathcal{B}^*\tilde{T}\mathcal{B}$, 
 gives a Bargmann-Fock space result from \cite{Bauer:2012}. For this to work, it is important that $\varphi_0\otimes \varphi_0 \in \mathcal{W}$, since Proposition \ref{prop:unitarilyequivalent} and Lemma \ref{lem:berezinfock} relate the Bargmann-Fock setting to convolutions with $\varphi_0\otimes \varphi_0.$ The definition of slowly oscillating functions on $\Rdd$ given after that theorem is adapted to $\mathbb{C}^d$ in an obvious way.

\bauerisralowitz

\begin{rem}
The last remark on slowly oscillating functions is, to our knowledge, a new contribution, and follows from Theorem \ref{thm:pitt}. However, we mention that there exist other results relating the behaviour of $F$ and $\mathcal{B}
^{\mathcal{F}^2}T_F^{\mathcal{F}^2}$  to the essential spectrum and Fredholmness of $T^{\mathcal{F}^2}_F$, also for classes of $F$ defined in terms of the oscillation\cite{Berger:1987,Al-Qabani:2018,Fulsche:2019hagger,Stroethoff:1992}. For instance, \cite[Thm. 33]{Fulsche:2019hagger} implies that slow oscillation could be replaced by \textit{vanishing oscillation} (see \cite{Fulsche:2019hagger} for the definition) in the theorem above, which is weaker as functions of vanishing oscillation are bounded and uniformly continuous.
\end{rem}
By Lemma \ref{lem:berezinfock} we immediately obtain the following compactness criterion.
\toeplitzchar

\begin{rem}
	One could also define the Berezin transform for Toeplitz operators on polyanalytic Bargmann-Fock spaces and relate it to convolutions with $\varphi_n\otimes \varphi_n$. However, we would not be able to apply Proposition \ref{thm:wernerpitt} to this case, as $V_{\varphi_n} \varphi_n$ always has zeros for $n\neq 0$.
\end{rem}

Finally, we note that Theorem \ref{thm:wernerpitt} gives a simple condition for compactness of localization operators in terms of the Gaussian $\varphi_0$.

\localizationchar
\begin{proof}
	Recall that $\mathcal{A}^{\psi_1,\psi_2}_f=f\star (\psi_2 \otimes \psi_1)$, so $\mathcal{A}^{\psi_1,\psi_2}_f \in \mathcal{C}_1$ by  \eqref{eq:toeplitzareuc}. Since $\varphi_0\otimes \varphi_0 \in \mathcal{W}$ by Example \ref{exmp:gaussianfw},  Proposition \ref{thm:wernerpitt} gives that $f\star (\psi_2 \otimes \psi_1)$ is compact if and only if $\left[f\star (\psi_2 \otimes \psi_1)\right]\star (\varphi_0 \otimes \varphi_0)\in C_0(\Rdd)$. The result therefore follows by
	\begin{align*}
  \left[f\star (\psi_2 \otimes \psi_1)\right]\star (\varphi_0 \otimes \varphi_0)&=f\ast \left[(\psi_2 \otimes \psi_1) \star (\varphi_0 \otimes \varphi_0)\right] \quad \text{ by associativity} \\
  &= f\ast (V_{\varphi_0}\psi_2\overline{V_{\varphi_0}\psi_1}) \quad \text{ by Lemma \ref{lem:rankonecase} as } \check{\varphi_0}=\varphi_0.
\end{align*}

\end{proof}

In a sense, this result complements Theorem \ref{thm:fernandezgalbis}. Theorem \ref{thm:fernandezgalbis} characterized those $f$ such that $\mathcal{A}^{\varphi_1,\varphi_2}_f=f\star (\varphi_2 \otimes \varphi_1)$ is compact for all non-zero windows $\varphi_1,\varphi_2 \in L^2(\Rd)$. Proposition \ref{prop:charcompact} gives a characterization of compactness of $\mathcal{A}^{\psi_1,\psi_2}_f$ for a particular pair of windows $\psi_1,\psi_2$. Of course, when $$\F_W(\psi_2 \otimes \psi_1)(x,\omega)=e^{i\pi x \cdot \omega}V_{\psi_1}\psi_2(x,\omega)$$ has no zeros, compactness of $\mathcal{A}^{\psi_1,\psi_2}_f$ implies compactness of $\mathcal{A}^{\varphi_1,\varphi_2}_f$ for all windows $\varphi_1,\varphi_2$ by picking $S=\psi_2\otimes \psi_1$ and $A=0$ in Theorem \ref{thm:tauberian1}.

\section{Quantization schemes and Cohen's class}% \label{sec:quantcohen}

The perspective of \cite{Luef:2018b} is that any $R\in \bo$ defines both a quantization scheme and a time-frequency distribution. The quantization scheme associated with $R$ -- by which we simply mean a map sending functions on phase space $\Rdd$ to operators on $L^2(\Rd)$ -- is given by \begin{equation*} 
  f\mapsto f\star R \quad \text{ for } f\in L^1(\Rdd).
\end{equation*}

The time-frequency distribution $Q_R$ associated with $R$ is given by sending $\psi \in L^2(\Rd)$ to its time-frequency distribution $$Q_R(\psi)(z) = [(\psi \otimes \psi) \star \check{R}](z) \quad \text{ for } z\in \Rdd.$$ Recall that a quadratic time-frequency distribution $Q$ is said to be of Cohen's class if there is some $a\in \mathscr{S}'(\Rdd)$ such that
\begin{equation} \label{eq:cohenclassfunctions}
  Q(\psi)=a\ast W(\psi, \psi) \quad \text{ for all } \psi\in \mathscr{S}(\Rd).
\end{equation}
The distribution $Q_R$ is of Cohen's class as \eqref{eq:opconvweyl} implies that
\begin{equation}\label{eq:cohenclassweyl}
  Q_R(\psi)= \weyl_{\check{R}}\ast W(\psi,\psi),
\end{equation}
 where $\weyl_{\check{R}}$ is the Weyl symbol of $\check{R}$. Using Theorem \ref{thm:tauberian2}, we deduce the following result relating compactness of the quantization scheme of $R$ to $C_0(\Rdd)$ membership of $Q_R$.

\quantization

\begin{proof}
	 Since $Q_R(\psi)(z)=\check{R}\star (\psi \otimes \psi)$ and $\F_W(\varphi\otimes \varphi)(x,\omega)=e^{i\pi x\cdot \omega}V_\varphi \varphi(x,\omega)$, it follows from Theorem \ref{thm:tauberian2} with $A=0$ that $(i)\iff (iii)$ and $(ii)\iff (iv).$ A short calculation shows that $R\star (\psi \otimes \psi)(z) = Q_R(\check{\psi})(-z).$ Since $\psi \mapsto \check{\psi}$ is a bijection on $L^2(\Rd)$, it follows that $(iii)$ is equivalent to $$(iii')\  R\star (\psi \otimes \psi)\in C_0(\Rdd) \ \text{ for all $\psi \in L^2(\Rd)$.}$$ By Theorem \ref{thm:tauberian2}, $(iii')\iff (iv)$, which finishes the proof. 
\end{proof}
\begin{rem}
\begin{enumerate}
	\item By the remark following Theorem \ref{thm:tauberian2}, the conditions on $\varphi$ in $(i)$ and $g$ in $(ii)$ are also necessary to imply $(iii)$ and $(iv)$. 
	\item One advantage of using the operator convolutions to describe Cohen's class is that $\psi\otimes \psi \in \tco$ for any $\psi \in L^2(\Rd)$, so as long as $R$ is a bounded operator we may exploit results on $\bo\star \tco$ to study $Q_R(\psi)=\check{R}\star (\psi \otimes \psi)$. If we had used the description of Cohen's class using functions in \eqref{eq:cohenclassfunctions}, one could similarly hope that $W(\psi,\psi)\in L^1(\Rdd)$, so that picking $a\in L^\infty(\Rdd)$ allows us to study $Q(\psi)=a\ast W(\psi,\psi)$ as convolutions of bounded and integrable functions. Unfortunately, $W(\psi,\psi)\in L^1(\Rdd)$ if and only if $\psi$ belongs to a proper subspace of $L^2(\Rd)$ called \textit{Feichtinger's algebra} \cite{Feichtinger:1981}. Hence this approach fails in general.
\end{enumerate}

\end{rem}
The gist of the above proposition is that $(i)$ provides a simple test for checking whether $(iii)$ and $(iv)$ hold. A typical choice for $\varphi$ in $(i)$ would be the Gaussian $\varphi = \varphi_0$, then $Q_R(\varphi_0)$ is the so-called \textit{Husimi function} of $R$. Hence the quantization $f\star R$ of any $f\in L^1(\Rdd)$ is compact and $Q_R(\psi)\in C_0(\Rdd)$ for any $\psi \in L^2(\Rd)$ if and only if the Husimi function of $R$ belongs to $C_0(\Rdd).$
\subsubsection{$\tau$-Wigner distributions}% \label{sec:tau}

For $\tau\in [0,1]$, define
\begin{equation*} %\label{eq:tauweyl}
	a_{\tau}(x,\omega)=\begin{cases}
		 \frac{2^d}{|2\tau-1|^d}\cdot e^{2\pi i \frac{2}{2\tau -1}x\cdot \omega}\quad &\text{ if } \tau \neq \frac{1}{2}, \\
		 \delta_0 \quad &\text{ if } \tau=\frac{1}{2},
	\end{cases}
\end{equation*}
where $\delta_0$ is Dirac's delta distribution. A slightly tedious calculation using the definition \eqref{eq:weyldef} shows that the Weyl transform $S_\tau$ of $a_{\tau}$ is given for $\psi\in \mathscr{S}(\Rd)$ by 
\begin{equation*}
  S_\tau(\psi)(t)=\begin{cases} \frac{1}{(1-\tau)^d}\psi\left(\frac{\tau}{\tau-1}\cdot t\right) \quad &\text{ if } \tau \in (0,1), \\
  	\psi(0) \quad &\text{ if } \tau = 0, \\
  	\int_{\Rd}\psi(t) \ dt \cdot \delta_0 \quad &\text{ if } \tau = 1,
   \end{cases} 
\end{equation*}
as already noted for $d=1$ in \cite[Thm. 7.2]{Koczor:2018}. If $\tau \in (0,1)$, it is easy to check that $S_\tau$ is bounded on $L^2(\Rd)$ with $\|S_\tau\|_{\bo}=\frac{1}{(1-\tau)^{d/2}\tau^{d/2}},$ that $S_\tau^* = S_{1-\tau}$, $\check{S}_\tau=S_\tau$ and the inverse of $S_\tau$ is $\tau^d (1-\tau)^d S_{1-\tau}$. In particular, $S_\tau$ is not compact. 

In light of \eqref{eq:cohenclassweyl}, \cite[Prop. 5.6]{Boggiatto:2010} states that $Q_{S_\tau}(\psi)$ is the $\tau$-Wigner distribution $W_\tau(\psi)$ introduced in \cite{Boggiatto:2010}, given explicitly by $$Q_{S_\tau}(\psi)(z)=W_\tau(\psi)(z):=\int_{\Rd}e^{-2\pi i t\cdot \omega}\psi(x+\tau t)\overline{\psi(x-(1-\tau)t)} dt.$$

On the other hand, we easily find for $f\in L^1(\Rdd)$ and $\psi\in \mathscr{S}(\Rd)$ that
\begin{align*}
	\inner{(f\star S_{1-\tau})\psi}{\psi}_{L^2}&=\left[(f\star S_{1-\tau})\star (\psi \otimes \psi)\right](0) \\
	&= \left[f\ast (S_{1-\tau}\star (\psi \otimes \psi))\right](0) \\
	&= \int_{\Rdd} f(z) S_{1-\tau}\star (\psi \otimes \psi) \ dz \\
	&= \int_{\Rdd} f(z) W_{1-\tau}(\psi)(z) \ dz \\
	&= \inner{f}{W_{\tau}(\psi)}_{L^2(\Rdd)}.
\end{align*}
In the last line we use that $\overline{Q_S(\psi)}=Q_{S^*}(\psi)$ for $S\in \bo$, and $S_\tau^*=S_{1-\tau}$. This shows precisely that $f\star S_{1-\tau}$ satisfies the definition of the $\tau$-Weyl quantization of $f$ introduced by Shubin \cite{Shubin:1987} -- in the notation of \cite{Boggiatto:2010} we have that $$f\star S_{1-\tau}=W^f_\tau.$$

The case $\tau=1/2$ is of particular interest, as $S_{1/2}=S_{1-1/2}=2^dP$ -- a scalar multiple of the parity operator. This case corresponds to the Weyl calculus, in the sense that $Q_{2^d P}(\psi)=W(\psi,\psi)$ for $\psi\in L^2(\Rd)$ and $f\star (2^d P)$ is the Weyl transform of $f$ for $f\in L^1(\Rdd).$

We can now show that the $\tau$-Wigner theory and the non-compact operators $S_\tau$ give a family of non-trivial examples to Theorem \ref{thm:tauberian2}. The compactness part of the next result was also noted using different methods in \cite[Thm. 6.9]{Boggiatto:2010}.

\begin{prop}
	Let $\tau \in (0,1)$. Then $S_\tau$ satisfies condition $(i)$ of Proposition \ref{prop:quantization}, hence
	\begin{enumerate}
		\item $W_\tau(\psi)=Q_{S_\tau}(\psi) \in C_0(\Rdd)$ for any $\psi \in L^2(\Rd)$.
		\item $W^f_\tau=f\star S_{1-\tau}$ is a compact operator on $L^2(\Rd)$ for any $f\in L^1(\Rdd).$
	\end{enumerate}
\end{prop}

\begin{proof}
	Recall from Example \ref{exmp:gaussianfw} that $V_{\varphi_0}\varphi_0$ has no zeros. By \cite[Prop. 4.4]{Boggiatto:2010},
	\begin{equation*}
	    Q_{S_\tau}(\varphi_0)=W_\tau(\varphi_0)\in C_0(\Rdd)
	\end{equation*} for any $\tau\in [0,1]$. Hence $(i)$ in Proposition \ref{prop:quantization} is satisfied, and the result follows by $(iii)$ and $(iv)$ of the same proposition.
\end{proof}

In fact, the same proof shows that the functions $a_\tau\in L^\infty(\Rdd)$ for $\tau \neq 1/2$ are non-trivial examples of Theorem \ref{thm:tauberian1}, where non-trivial refers to the fact $\weyl_\tau \notin L^p(\Rdd)$ for $p=0$ or $1\leq p < \infty$.

\begin{prop}\label{prop:tauwienerexample}
	For $\tau\in [0,1]\setminus \left\{\frac{1}{2}\right\}$, $a_\tau$ satisfies the assumptions of Theorem \ref{thm:tauberian1} with $A=0$.
\end{prop}

\begin{proof}
	Recall that $W_\tau(\varphi_0)=S_{\tau} \star (\varphi_0 \otimes \varphi_0)=a_\tau \ast W(\varphi_0,\varphi_0)$ by \eqref{eq:cohenclassweyl}. As a special case of \eqref{eq:weyltransformfw} one gets that $\F_\sigma(W(\varphi_0,\varphi_0))=\F_W(\varphi_0\otimes \varphi_0)$, hence $W(\varphi_0,\varphi_0)\in W(\Rdd)$ by Example \ref{exmp:gaussianfw}. The previous proof showed that $W_\tau(\varphi_0)\in C_0(\Rdd)$, so $f$ satisfies assumption $(ii)$ of Theorem \ref{thm:tauberian1}.
\end{proof}

\begin{rem} \label{rem:weylboundedness}
 The operators $S_0$ and $S_1$ are clearly not bounded on $L^2(\Rd)$, even though $\weyl_0,\weyl_1\in L^\infty(\Rdd)$. Hence $\weyl_0$ and $\weyl_1$ are examples of bounded functions with unbounded Weyl transform. Similarly, $S_{1/2}$ is a bounded operator with unbounded Weyl symbol. 
 \end{rem}

We end by considering the example of Born-Jordan quantization. 
\begin{exmp}[Born-Jordan quantization]

The \textit{Born-Jordan distribution} $Q_{BJ}(\psi)$ of $\psi \in L^2(\Rd)$ is given by $$Q_{BJ}(\psi)(z)=\int_{0}^1 W_\tau(\psi)(z) \ d\tau=\int_{0}^1 Q_{S_\tau}(\psi)(z) \ d\tau,$$ see \cite{Boggiatto:2010,deGosson:2016}. It is well-known that $Q_{BJ}$ is of Cohen's class, and from \cite[Prop. 5.8]{Boggiatto:2010} it follows that $Q_{BJ}=Q_{S_{BJ}}$ where $S_{BJ}\in \mathcal{L}(\mathscr{S}(\Rd),\mathscr{S}'(\Rd))$ is defined by 
\begin{equation} \label{eq:fwboundedbj}
\F_W(S_{BJ})(x,\omega)=\mathrm{sinc}(\pi x\cdot \omega).
\end{equation}
The associated quantization scheme $f\mapsto f\star S_{BJ}$ is then the \textit{Born-Jordan quantization} \cite{deGosson:2016}.

For $d=1$ it was shown in \cite[Prop. 2]{Koczor:2018} that $S_{BJ}\in \bo$. Since \eqref{eq:fwboundedbj} shows that $\F_W(S_{BJ})\in L^\infty(\Rdd)$, combining Example \ref{exmp:operatorpseudomeasures} and Proposition \ref{prop:quantization} we may conclude that the Born-Jordan quantization of any $f\in L^1(\R^2)$ is compact, and that the Born-Jordan distribution of any $\psi\in L^2(\R)$ belongs to $C_0(\R^2).$

% To see that $S_{BJ}\in \bo$ for $d=1$, we note that for $\psi \in \mathscr{S}(\Rd)$ we find
% \begin{align*}
% 	|\inner{S_{BJ}\psi}{\psi}_{L^2}|&=|Q_{BJ}(\psi,\psi)(0)| \quad \text{ by Lemma \ref{lem:rankonecase} }\\
% 	&\leq \int_0^1|Q_{S_\tau}(\psi)(0)| \ d\tau \\
% 	&\leq \|\psi\|_{L^2}^2 \int_0^1 \|S_\tau\|_{\bo} \ d\tau \\
% 	&\leq \|\psi\|_{L^2}^2 \int_0^1\frac{1}{\tau^{d/2}(1-\tau)^{d/2}} d\tau \\
% 	&= \begin{cases}
% 		 \pi \quad &\text{ if } d=1\\
% 		 \infty \quad &\text{otherwise}.
% 	\end{cases}
% \end{align*}
% The quadratic form $Q$ is therefore bounded if $d=1$, so in this case there must exist some $S_{BJ}\in \bo$ with $Q(\psi)=\inner{S_{BJ}\psi}{\psi}_{L^2}$. One then checks that $$\inner{S_{BJ}\pi(z)^*\psi}{\pi(z)^* \psi}_{L^2}=Q(\pi(z)^*\psi))=Q_{BJ}(\pi(z)^*\psi)(0)=Q_{BJ}(\psi)(z),$$ which by Lemma \ref{lem:rankonecase} implies that $Q_{BJ}(\psi)=S_{BJ}\star (\psi \otimes \psi)$ and $Q_{BJ}=Q_{S_{BJ}}$.  As $\F_W(R)\in L^\infty(\R^2)$  by definition,  The boundedness of $S_{BJ}$ was previously noted with a different proof in \cite[Prop. 2]{Koczor:2018}. 

 \end{exmp}

\subsection{Counterexample to a Schatten class version of Theorem \ref{thm:tauberian2}}

For the special case $A=0$, Theorem \ref{thm:tauberian2} states that if $R\star a\in \compacts$ for some $a\in W(\Rdd)$, then $R\star g \in \compacts$ for all $g\in L^1(\Rdd).$ An obvious generalization is to replace $\compacts$ by a Schatten class $\SC^p$ for some $1\leq p <\infty$. Is it true that $R\star a \in \SC^p$ for $a \in W(\Rdd)$ implies that $R\star g \in \SC^p$ for all $g\in L^1(\Rdd)$? A simple counterexample is provided by the Weyl calculus.

\begin{exmp}
	Recall that $S_{1/2}\star f$ is the Weyl transform of $f\in L^1(\Rdd).$ If we let $a(z)=2^d e^{-\pi |z|^2}$, then $a\in W(\Rdd)$ and it is well-known that the Weyl transform $a\star S_{1/2}$ of $a$ is the rank-one operator $\varphi_0 \otimes \varphi_0$. In particular, $a\star S_{1/2} \in \tco \subset \SC^p$ for any $1\leq p \leq \infty$. However, if we pick $f\in L^1(\Rdd)\setminus L^2(\Rdd)$, then $f\star S_{1/2}\notin \SC^p$ for any $1\leq p \leq 2$, since the Weyl transform is a unitary mapping from $L^2(\Rdd)$ to $\HS$, and $\SC^p\subset \HS$ for $1\leq p \leq 2.$ Hence we cannot conclude from $a\star S_{1/2}\in \SC^p$ for $a\in W(\Rdd)$ that $f\star S_{1/2}\in \SC^p$ for all $f\in L^1(\Rdd)$, at least for $1\leq p \leq 2$. 
\end{exmp}

  \section*{Acknowledgements} 
  We wish to thank Alexander Bufetov for asking one of the authors about the existence of Tauberian theorems in quantum harmonic analysis, which prompted the writing of this paper. We also thank Hans Feichtinger for pointing out the result from reference \cite{Braun:1983} mentioned in Example \ref{exmp:tauberian1}, and Eirik Berge for helpful discussion on Gabor spaces.

\bibliographystyle{plain}         % Style BST file
   \bibliography{kilder}{} 
\end{document}